\documentclass[3p,,preprint,12pt]{elsarticle}

\usepackage{titlecaps}
\Addlcwords{the of into via for and of on in an to hp-finite with}
\usepackage{tabulary,xcolor}
\usepackage{amsfonts,amsmath,amssymb}
\usepackage[T1]{fontenc}
\makeatletter
\let\save@ps@pprintTitle\ps@pprintTitle
\def\ps@pprintTitle{\save@ps@pprintTitle\gdef\@oddfoot{\footnotesize\itshape \null\hfill\today}}
\def\hlinewd#1{%
  \noalign{\ifnum0=`}\fi\hrule \@height #1%
  \futurelet\reserved@a\@xhline}

\AtBeginDocument{\ifNAT@numbers \biboptions{sort&compress}\fi}
\makeatother

\usepackage{ifluatex}
\ifluatex
\usepackage{fontspec}
\defaultfontfeatures{Ligatures=TeX}
\usepackage[]{unicode-math}
\unimathsetup{math-style=TeX}
\else 
\usepackage[utf8]{inputenc}
\fi 
\ifluatex\else\usepackage{stmaryrd}\fi

\usepackage{url,multirow,morefloats,floatflt,cancel,tfrupee}
\makeatletter

\AtBeginDocument{\@ifpackageloaded{textcomp}{}{\usepackage{textcomp}}}
\makeatother
\usepackage{colortbl}
\usepackage{xcolor}
\usepackage{pifont}
\usepackage[nointegrals]{wasysym}
\makeatletter

\def\mcWidth#1{\csname TY@F#1\endcsname+\tabcolsep}

\def\cAlignHack{\rightskip\@flushglue\leftskip\@flushglue\parindent\z@\parfillskip\z@skip}
\def\rAlignHack{\rightskip\z@skip\leftskip\@flushglue \parindent\z@\parfillskip\z@skip}

\if@twocolumn\usepackage{dblfloatfix}\fi 
\AtBeginDocument{
\expandafter\ifx\csname eqalign\endcsname\relax
\def\eqalign#1{\null\vcenter{\def\\{\cr}\openup\jot\m@th
  \ialign{\strut$\displaystyle{##}$\hfil&$\displaystyle{{}##}$\hfil
      \crcr#1\crcr}}\,}
\fi
}

\let\lt=<
\let\gt=>
\def\processVert{\ifmmode|\else\textbar\fi}

\@ifundefined{subparagraph}{
\def\subparagraph{\@startsection{paragraph}{5}{2\parindent}{0ex plus 0.1ex minus 0.1ex}%
{0ex}{\normalfont\small\itshape}}%
}{}

\newcommand\role[1]{\unskip}
\newcommand\aucollab[1]{\unskip}
  
\@ifundefined{tsGraphicsScaleX}{\gdef\tsGraphicsScaleX{1}}{}
\@ifundefined{tsGraphicsScaleY}{\gdef\tsGraphicsScaleY{.9}}{}
\def\checkGraphicsWidth{\ifdim\Gin@nat@width>\linewidth
    \tsGraphicsScaleX\linewidth\else\Gin@nat@width\fi}

\def\checkGraphicsHeight{\ifdim\Gin@nat@height>.9\textheight
    \tsGraphicsScaleY\textheight\else\Gin@nat@height\fi}

\def\fixFloatSize#1{}
\let\ts@includegraphics\includegraphics

\def\inlinegraphic[#1]#2{{\edef\@tempa{#1}\edef\baseline@shift{\ifx\@tempa\@empty0\else#1\fi}\edef\tempZ{\the\numexpr(\numexpr(\baseline@shift*\f@size/100))}\protect\raisebox{\tempZ pt}{\ts@includegraphics{#2}}}}

\AtBeginDocument{\def\includegraphics{\@ifnextchar[{\ts@includegraphics}{\ts@includegraphics[width=\checkGraphicsWidth,height=\checkGraphicsHeight,keepaspectratio]}}}

\def\URL#1#2{\@ifundefined{href}{#2}{\href{#1}{#2}}}

\def\UrlOrds{\do\*\do\-\do\~\do\'\do\"\do\-}%
\g@addto@macro{\UrlBreaks}{\UrlOrds}
\makeatother

\usepackage{pgfplots,algorithmic,algorithm}
\usepackage{booktabs}
\usepackage[frozencache]{minted}
\usepackage[toc,page]{appendix}
\usepackage[toc,page]{appendix}

\usepackage[textsize=small]{todonotes}

\usepackage{pifont}

\journal{Journal Name}
\usepackage{hyperref}
\usepackage[nameinlink]{cleveref}
\usepackage{autonum}

\newcommand{\ii}[0]{\mathrm{i}}
\newcommand{\RR}[0]{\mathbb{R}}
\newcommand{\bx}[0]{\mathbf{x}}
\newcommand{\bw}[0]{\mathbf{w}}
\newcommand{\bu}[0]{\mathbf{u}}

\newcommand{\bA}[0]{\mathbf{A}}
\newcommand{\bbf}[0]{\mathbf{f}}
\newcommand{\lib}[1]{\textcolor{blue}{\section{#1}}}
\usepackage{amsthm}
\theoremstyle{definition}
\newtheorem{definition}{Definition}[section]

\newcommand{\by}[0]{\mathbf{y}}

\usepackage{amsthm}
\newtheorem{lemma}{Lemma}
\newtheorem{theorem}{Theorem}

\newtheorem{remark}{Remark}

\begin{document}

\begin{frontmatter}

\title{Efficient Numerical Method for Models Driven by L\'evy Process via Hierarchical Matrices}
\author[label1]{Kailai Xu}
 \ead{kailaix@stanford.edu}
 \address[label1]{Institute for Computational and Mathematical Engineering, Stanford University, Stanford, CA, 94305}
%
 \author[label2]{Eric Darve}
 \ead{darve@stanford.edu}
 \address[label2]{Mechanical Engineering, Stanford University, Stanford, CA, 94305}

\begin{abstract}
     Modeling via fractional partial differential equations or a L\'evy process has been an active area of research and has many applications. However, the lack of efficient numerical computation methods for general nonlocal operators impedes people from adopting such modeling tools. We proposed an efficient solver for the convection-diffusion equation whose operator is the infinitesimal generator of a L\'evy process based on $\mathcal{H}$-matrix technique. The proposed Crank Nicolson scheme is unconditionally stable and has a theoretical $\mathcal{O}(h^2+\Delta t^2)$ convergence rate. The $\mathcal{H}$-matrix technique has theoretical $\mathcal{O}(N)$ space and computational complexity compared to $\mathcal{O}(N^2)$ and $\mathcal{O}(N^3)$ respectively for the direct method. Numerical experiments demonstrate the efficiency of the new algorithm. 
\end{abstract}

    \begin{keyword}
L\'evy Process \sep Hierarchical Matrices \sep Fractional Partial Differential Equation


\end{keyword}

\end{frontmatter}

\lib{Introduction}
Over the last years anomalous diffusion or nonlocal modeling have seen a tremendous increase in popularity in many fields. Of particular interest is the fractional partial differential equations~(FPDE) arising from many disciplines such as image processing~\cite{gatto2015numerical,unser2009multiresolution}, finance~\cite{scalas2000fractional}, stochastic dynamics~\cite{bogdan2003censored}, fractional kinetics and anomalous transport~\cite{zaslavsky2002chaos}, fractal conversation laws~\cite{alibaud2012continuous}, fluid dynamics~\cite{chen2010anomalous,chen2006speculative,epps2018turbulence}, and so on.  One extensively studied fractional operator is the fractional Laplacian~\cite{lischke2018fractional}
\begin{equation}
    -(-\Delta)^s u(\bx):=c_{d,s}\mathrm{p.v.}\int_{\RR^d} \frac{u(\bx+\by)-u(\bx)}{|\by|^{d+2s}}d\by\quad c_{d,s} := \frac{4^s\Gamma(d/2+s)}{\pi^{d/2}|\Gamma(-s) |}
\end{equation}
which is considered as a generalization of the Laplacian operator. Here $\mathrm{p.v.}$ denotes the principal value integration. In \Cref{sect:app} we list several applications of the fractional Laplacian operator in finance, quantum mechanics and turbulence flow. However, the numerical computation of the FPDE with such operators exhibits special difficulties~\cite{bonito2018numerical,huang2016finite}: (1) the kernel function $\frac{c_{d,s}}{|\by|^{d+2s}}$ has singularities which must be dealt with special care; (2) the kernel function is nonlocal, and therefore the corresponding coefficient matrix is typically dense. The second difficulty impedes people from using the new modeling tool due to its prohibitive computational requirement. There are some efforts to speed up the computation~\cite{meidner2017hp,kyprianou2017unbiased,acosta2018regularity,zhao2017adaptive}, mainly through analyzing its special structure or modifying the definition. 

From another point of view, the fractional partial differential equation with the Laplacian operator~(and many others) can be derived from the infinitesimal generator of the L\'evy process. In particular, the fractional Laplacian corresponds to a symmetric stable process~\cite{garbaczewski2018fractional}. Indeed, in 1D, the forward equation~(or Fokker Planck equation in physics) has the form~\cite{barndorff2012levy}
\begin{equation}\label{equ:levy}
    u_t = a u_{xx} + bu_x + cu + \mathcal{L}u\quad x\in \RR, t\in (0,1)
\end{equation}
where $a\geq 0$, $c\leq 0$, $b\in\RR$, and
\begin{equation}
    \mathcal{L}u = \int_\RR \left(u(x+y)-u(x)-u'(x)\mathbf{1}_{0<|y|<1}(y)y\right)\nu(y)dy
\end{equation}  
Here $\nu(y)$ will be a proper L\'evy measure. For more details on how \cref{equ:ut} naturally arises from L\'evy process, see \Cref{sect:levy}. For some concrete applications, see \Cref{sect:app}. The fractional Laplacian is a special case where $\nu(y) = \frac{c_{1,s}}{|y|^{1+2s}}$~\cite{kwasnicki2017ten}. The model \cref{equ:levy} incorporates a much richer structure and has a broader of applications. For example, in recent years, the modeling of financial markets by L\'evy processes has become an active area of research~\cite{MichaelC25:online}. The numerical difficulties are similar to that of FPDE. 


In this paper, we aim at solving \cref{equ:levy} efficiently based on the well-established $\mathcal{H}$-matrix technique~\cite{bebendorf2008hierarchical,borm2003introduction}. In principle, our algorithm can work for various $\nu(y)$ under mild assumptions, including singular or slow decaying L\'evy measure. We focus on the efficiency of the operator since the ability to efficiently store data and solve is the main bottleneck for today's applications. In particular, the algorithm will equivalently work for many FPDE models, on condition it can be written in the form of \cref{equ:levy}. 

The advantage of adopting the $\mathcal{H}$-matrix is its high efficiency. If direct method is used, which results in a dense coefficient matrix, the storage complexity will be $\mathcal{O}(N^2)$ while the computational complexity will be $\mathcal{O}(N^3)$~(LU factorization)~\cite{golub2012matrix}. However, theoretically, $\mathcal{H}$-matrix can achieve nearly optimal $\mathcal{O}(N)$ storage and computational complexity~\cite{borm2003introduction,bebendorf2008hierarchical}. Similar efforts for efficiently tackling nonlocal problems include application of FFT to circulant or Toeplitz-like stiffness matrix~\cite{hatzinikitas2009fractional,chen2016two,bailey1994fast,chaturapruek2013crime}; however, these methods are restricted to shift-invariant discretization, which usually requires uniform grids and constant coefficients in PDE. Another direction is the use of hierarchical matrices, which we will pursue in the paper. For example, \cite{zhao2017adaptive} adopted adaptive finite element method for FPDEs using hierarchical matrices in 1D; \cite{massei2018fast} analyzed the use of HOLDER arithmetic for solving the 1D case and leveraged the properties to design fast solvers for 2D problems; \cite{karkulikh} used a Galerkin approximation based on piecewise linear functions on a quasi-uniform mesh to the fractional Laplacian on a bounded domain and showed that the inverse of the associated stiffness matrix can be approximated by the block-wise low-rank matrices at an exponential rate in the block rank. Our $\mathcal{H}$-matrix algorithm is distinguished from the existing work in the following ways: the construction and LU-factorization of the $\mathcal{H}$-matrix is completely automatic. In the series expansion version, the users only need to specify the kernel functions and its low-rank expansion series, and then the algorithms will find an optimal $\mathcal{H}$-matrix structure and LU factorize it. In the Blackbox FMM version, the users do not even need to specify the low-rank expansion. The users can also provide the corresponding dense matrix and our algorithm will automatically figure out the corresponding reordering and $\mathcal{H}$-matrix structure.
	
 The new algorithm shows great speedup compared to the direct method for medium and large-scale problems~(for example, in 1D, the crossover for LU, which is the most expensive operation, is around $N=1100$). 

We mention that there are other approaches to solve FPDE. One of the main numerical methods is the Monte Carlo methods~\cite{metropolis1989monte}, which is based on the probabilistic interpretation of the model. If $X_t$ is a L\'evy process with the L\'evy measure $\nu(y)$ and appropriate diffusion and drift coefficients, under certain assumptions, the solution to \cref{equ:levy} can be written as~\cite{MichaelC25:online}
\begin{equation}
    u(x) = \mathbb{E}(u(X_t)|X_0=x)
\end{equation}
So a Monte Carlo method can be applied thereafter. Although Monte Carlo might be the only way possible to compute the solution in high dimensions numerically, it suffers from slow convergence and therefore is impractical for some cases~\cite{hammersley2013monte}. The grid-based method, such as the one we proposed in the paper, will enjoy fast convergence~(and we will prove that the convergence order is $\mathcal{O}(\Delta t^2+h^2)$). 

To end this section, we summarize our major contributions of the paper
\begin{itemize}
    \item Proposed and analyzed an unconditional stable Crank Nicolson scheme for the model problem \cref{equ:levy}. The theoretical error is $\mathcal{O}(\Delta t^2 + h^2)$. For the variable fractional index case where the computational domain is truncated, we show empirically that the error rate is reduced to $\mathcal{O}(\Delta t^2 + h)$.

    \item Proposed and implemented an efficient solver for \cref{equ:levy} based on  $\mathcal{H}$-matrix techniques. The memory and computational complexity is $\mathcal{O}(N)$ if the kernel satisfies some regularity properties~(see \Cref{sect:hmat} for details). 
    \item Proposed a method for computing nonlocal operators involving L\'evy measures that are singular and have a heavy tail~(decay slowly).
    \item Solved a variable index space-fractional Poisson problem on a L-shaped domain using the proposed algorithm.
\end{itemize}

\lib{Applications \label{sect:app}}

In this section, we list several possible applications of the numerical scheme and fast algorithms. These applications are taken from literature which can be formulated as an integrodifferential equation.

\subsection{Option Pricing}

One of the applications of the L\'evy process modeling is the option pricing, where the underlying asset price is assumed to follow a L\'evy process instead of the Brownian motion~\cite{tankov2009jump,kou2002jump}.

Let $S_t$ be the price of a financial asset which is modeled as a stochastic process under a martingale equivalence measure $\mathbb{Q}'$ and on a filtered probability space $(\Omega, \mathcal{F}, \mathcal{F}_t, \mathbb{Q})$.

One of the popular models is the exponential L\'evy model which assumes
\begin{equation}
    S_t = S_0 e^{rt+X_t}
\end{equation}
where $X_t$ is a L\'evy process. Assume $r$ is the interest rate. For a European call or put, the terminal payoff $H_T$ at time $T$ is associated with the underlying asset price $S_T$
\begin{equation}
    H_T = H(S_T)
\end{equation}
 The value of the option is defined as a discounted conditional expection of $H_T$ under the risk-adjusted martingale measure
\begin{equation}
    C_t = \mathbb{E}[e^{-r(T-t)}H(S_T)|\mathcal{F}_t] = \mathbb{E}[e^{-r(T-t)}H(S_T)|S_t=S]
\end{equation}

By introducing $\tau = T-t$, $x=\log\left( \frac{S}{S_0} \right)$, and define
\begin{equation}
    u(x,\tau) = \mathbb{E}[h(x+Y_\tau)]\quad h(x) = H(S_0 e^x)
\end{equation}
 for sufficiently smooth $u$, by applying the Ito's formula for L\'evy process we have the integro-differential equation
\begin{equation}
    \frac{\partial u}{\partial \tau} = \frac{\sigma^2}{2}u_{xx} - (\sigma^2/2-r+\alpha) u_x + \int_\RR (u(x+y)-u(x)-u'(x)\mathbf{1}_{0<|y|<1}(y)y)\nu(y)dy
\end{equation}
with initial condition
\begin{equation}
    u(0,x) = h(x)
\end{equation}

\subsection{Quantum Mechanics}

If the underlying stochastic process powering the random fluctuations is a Gaussian Brownian motion, we obtain the non relativistic Schr\"odinger's equation~\cite{laskin2010principles,hasan2018tunneling,garbaczewski1995schrodinger,laskin2000fractional}
\begin{equation}
    \ii \hbar \partial_t \psi(x,t) = -\frac{\hbar^2}{2m}\partial^2_{x} \psi(x,t)
\end{equation}

In recent years, there is a growing interest in the non-Gaussian stochastic process, and particularly the L\'evy process. One of the popular models is the fractional quantum mechanics, where the stable processes are used as the underlying stochastic process. The popularity of the stable process is justified by the properties of scaling and self-similarity displayed by the process. For any distribution with power-law decay $\frac{1}{|x|^{1+\alpha}}$, $0<\alpha<1$ the generalized central limit theorem guarantee that their sum scaled by $\frac{1}{n^{1/\alpha}}$ converge to the $\alpha$-stable distribution. If the variance is finite, i.e., $\alpha\geq 2$, then the central limit theorem holds, where their sum scaled by $\frac{1}{n^{1/2}}$, properly centered, and identically distributed, converge to the Gaussian distribution.  This leads to the fractional Schr\"odinger equation
\begin{equation}
    \ii \hbar \partial_t \psi(x,t) = D_\alpha (-\hbar^2 \Delta)^{\alpha/2}\psi(x,t)
\end{equation}
where $(-\hbar^2 \Delta)^{\alpha/2}$ is the fractional Laplacian which can be defined through
\begin{equation}
    (-\hbar^2 \Delta)^{\alpha/2}\psi(x,t) = \frac{1}{(2\pi \hbar)^3}\int |\xi|^\alpha \hat \psi(\xi,t) \exp(\ii (\xi,x)/\hbar)d\xi
\end{equation}

More generally, other L\'evy measures can be used to develop quantum mechanics. The more general Schr\"odinger equation reads
\begin{equation}
    \ii \hbar\partial_t \psi(x,t) = -\frac{\hbar^2}{2m}\partial_x^2 \psi(x,t) - \hbar  \int_\RR [\psi(x+y,t)-\psi(x,t)]\nu(y)dy
\end{equation}

Some examples of the L\'evy-Schr\"odinger equations are
\begin{itemize}
    \item Relativistic.
    \begin{equation}
        \ii \hbar\partial_t \psi(x,t) = \sqrt{m^2c^4-c^2\hbar^2\partial_x^2} \psi(x,t)
    \end{equation}
    \item Variance-Gamma laws
    \begin{equation}
        \ii \hbar\partial_t \psi(x,t) = -\frac{\lambda\hbar}{\tau}\int_\RR\frac{\psi(x+y,t)-\psi(x,t)}{|y|}e^{-|y|/\hbar}dy
    \end{equation}

\end{itemize}

\subsection{Turbulence Flow}

It is known that turbulence flow exhibits anomalous diffusion, i.e., the diffusion occurs over distance $\xi$ may scale more than one half, $\xi\sim\mathcal{O}(t^{1/2})$. There are many efforts to model turbulence and capture these anomalies~\cite{chen2010anomalous,chen2006speculative}. One of the recent research is the modeling of turbulence flow via the fractional Laplacian~\cite{epps2018turbulence}. 

If we assume that the equilibrium probability distribution of particle speeds to be L\'evy $\alpha$-stable distributions instead of the Maxwell-Boltzmann distribution, we will arrive at the Navier-Stokes equation with the fractional Laplacian operator as a means to represent the mean friction force arising in a turbulence flow
\begin{equation}
    \rho\frac{D\bar u}{Dt} = -\nabla p + \mu_\alpha \nabla^2 \bar u + \rho C_\alpha \int_{\RR^3} \frac{\bar u(x',t)-\bar u(x,t)}{|x-x'|^{\alpha+3}} dx'
\end{equation}

\lib{Crank Nicolson Scheme Based on $\mathcal{H}$-matrix}

 \subsection{Model Problem}
 We will consider the forward or backward equation driven by the L\'evy process, where the model problem in 1D can be stated as a convection-diffusion integrodifferential equation~\cite{MichaelC25:online}
\begin{equation}\label{equ:ut}
    u_t = a u_{xx} + bu_x + cu + \mathcal{L}u\quad x\in \RR, t\in (0,1)
\end{equation}
where $a\geq 0$, $c\leq 0$, $b\in\RR$, and
\begin{equation}
    \mathcal{L}u = \int_\RR (u(x+y)-u(x)-u'(x)\mathbf{1}_{0<|y|<1}(y)y)\nu(y)dy
\end{equation}

 \subsection{Numerical Scheme}

We consider the case where $\nu(y)<\infty$ and thus the term $u'(x)\mathbf{1}_{0<|y|<1}(y)y$ is not needed since
\begin{align}
	& \int_\RR (u(x+y)-u(x)-u'(x)\mathbf{1}_{0<|y|<1}(y)y)\nu(y)dy \\
	= &\int_\RR (u(x+y)-u(x))\nu(y)dy-u'(x)\int_\RR \mathbf{1}_{0<|y|<1}(y)ydy\\
	=& \int_\RR (u(x+y)-u(x))\nu(y)dy
\end{align}
due to symmetry of $\mathbf{1}_{0<|y|<1}(y)y$ around $y=0$. 

Also, we assume $\nu(y)$ is semi-heavy, i.e., there exists $\alpha_r,\alpha_l>0$, such that $\int_1^\infty e^{(1+\alpha_r)y}\nu(dy)<\infty$, and $\int_{-\infty}^{-1}|y|e^{\alpha_l|y|}\nu(dy)<\infty$.   The case for which $\nu(y)$ might grow to infinity at $y=0$ and decays algebraically will be discussed in \Cref{sect:extension}. To compute the integral term numerically, we need to restrict the computational domain to a bounded interval $\Omega$
\begin{equation}
    \mathcal{L}u \approx \int_{B_l}^{B_r} (u(x+y)-u(x)-u'(x)\mathbf{1}_{0<|y|<1}(y)y)\nu(y)dy
\end{equation}

In fact, it is proved in \cite{cont2005finite} that if $\nu(dy)$ is semi-heavy, the solution $\tilde u(x,t)$ obtained using the truncated integral will satisfy
\begin{equation}\label{equ:small}
    |u( x,t)-\tilde u(x,t)| =\mathcal{O}(e^{-\alpha_l|B_l|}+e^{-\alpha_r|B_r|})
\end{equation}
Therefore, the discretization scheme for $\mathcal{L}$ using trapezoidal rule on uniform grid will be
\begin{equation}
    (\mathcal{L} u)(jh) \approx   \sum_{j\in \mathcal{I}} u_{i+j}\nu_j w_j - u_i \lambda w_j\quad \lambda = \sum_{j\neq 0, j\in \mathcal{I}} \nu_j
\end{equation}
Here $\nu_j=\nu(jh)$, 
\begin{equation}
    \mathcal{I} = \{i: ih\in \Omega\}
\end{equation}
$w_j$ is the weight for the trapezoidal rule and 
\begin{equation}
    w_j = \begin{cases}
        h & j \mbox{ is not the endpoint of } \mathcal{I}\\
        \frac{h}{2} & j \mbox{ is the endpoint of } \mathcal{I}
    \end{cases}
\end{equation}
We define the discrete operator $\delta_L$
\begin{equation}\label{equ:dl}
    (\delta_L u)_j = \sum_{j=-\infty}^\infty (u_{i+j}-u_i)\nu_j h =  \sum_{j\in \mathcal{I}, j=-\infty}^\infty u_{i+j}\nu_j h - u_i \lambda h\quad \lambda = \sum_{j\neq 0, j\in \mathcal{I}} \nu_j
\end{equation}
Then the Crank-Nicolson discretization of \cref{equ:ut} on a uniform grid with spacing $h$ and timestep $\Delta t$ is 
\begin{equation}\label{equ:scheme}
    (I+\frac{1}{2}\Delta tA)u^{n+1} = (I-\frac{1}{2}\Delta tA)u^n
\end{equation}
where 
\begin{equation}
    A = -a\delta_x^2 - b\delta_{2x} - c - \delta_L
\end{equation}
here $\delta_x^2$ and $\delta_{2x}$ are the standard second difference and central first difference. 
Therefore, we have
\begin{equation}\label{equ:Aij}
  {A_{ij}} = 
  \begin{cases}
      {\frac{{2a}}{{{h^2}}} - c + \lambda w_{j-i}} & i=j\\
{ - \frac{a}{{{h^2}}} + \frac{b}{{2h}} - {\nu _{ - 1}w_{j-i}}} & j=i-1\\
{ - \frac{a}{{{h^2}}} - \frac{b}{{2h}} - {\nu _1}w_{j-i}} & j=i+1\\
{ - {\nu _{j - i}w_{j-i}}} & |j-i|\geq 2
  \end{cases}
\end{equation}
 
 \subsection{$\mathcal{H}$-matrix Construction} 
 
 For simplicity, assume $a=b=c=0$; according to \cref{equ:Aij}, these coefficients only contribute to the first off-diagonal parts of the coefficient matrix. We consider the matrix $A_\pm = I \pm \frac{1}{2}A$. Note since the operator $\delta_x^2$, $\delta_{2x}$ only contributes to the tridiagonal, any nonzero entry in $A_\pm$ in the off-diagonal more than one entry away from the diagonal must be $\mp\frac{1}{2}\nu_jh$ according to \cref{equ:dl}. 

We define the kernel associated with each L\'evy measure by
\begin{equation}
    K(x,y) = \nu(y-x)
\end{equation}
then we have $A_{ij}=\nu(x_{i}-x_{j})=K(x_i,x_j)$ for $|i-j|\geq 2$.

We illustrate here the application of the $\mathcal{H}$-matrix technique using the example of $A$ generated by the Gaussian kernel. For more details on the topic of the hierarchical matrices, see \Cref{sect:hmat}. Assume $a=c=0$ in \cref{equ:Aij}, then we can see that
 \begin{equation}
     A_{ij} = h^d k(x_i, x_j)
 \end{equation}
 for some kernel function $k(x,y)$. 
 
Consider the jump diffusion model with Gaussian jumps\footnote{It is also called Merton jump diffusion model in finance, see \cite{matsuda2004introduction,kou2002jump}}, i.e., the L\'evy density can be represented as
\begin{equation}
    \nu(x) = e^{-\varepsilon^2 x^2}
\end{equation}

We consider the kernel function $k(x,y)$ associated with the density
\begin{equation}
    k(x,y) = \nu(x-y) = e^{-\varepsilon^2 (x-y)^2}
\end{equation}
Assume that $x\in \mathcal{X}$, $y\in \mathcal{Y}$, and $\mathcal{X} \cap \mathcal{Y}=\emptyset$, and let $\bar x\in\mathcal{X}$. Denote $t_0=x-\bar x$, and $t=y-\bar x$, then by assumption we have $|t|>|t_0|$. From Taylor expansion we have
\begin{align}
    {e^{ - {\varepsilon ^2}{{(x - y)}^2}}} =& {e^{ - {\varepsilon ^2}{{(t - {t_0})}^2}}}\\
    =&  {e^{ - {\varepsilon ^2}{t^2}{{\left( {1 - \frac{{{t_0}}}{t}} \right)}^2}}} = {e^{ - {\varepsilon ^2}{t^2} - {\varepsilon ^2}t_0^2 + 2{\varepsilon ^2}{t_0}t}}\\
    =& {e^{ - {\varepsilon ^2}{t^2} - {\varepsilon ^2}t_0^2}}\left( {1 + 2{\varepsilon ^2}{t_0}t + \frac{{{{(2{\varepsilon ^2}{t_0}t)}^2}}}{2} + \frac{{{{(2{\varepsilon ^2}{t_0}t)}^3}}}{{3!}} +  \ldots } \right)
\end{align} 
Thus we have
\begin{equation}
    \alpha_n(t) = \frac{{{2^n}{\varepsilon ^{2n}}{e^{ - {\varepsilon ^2}t^2}}t^n}}{{n!}}\quad \beta_n(t) = {e^{ - {\varepsilon ^2}{t^2}}}{t^n}
\end{equation}
we will have
\begin{equation}
    {e^{ - {\varepsilon ^2}{{(x - y)}^2}}} =\sum_{n=0}^\infty \alpha_n(t_0)\beta_n(t)
\end{equation}

\begin{lemma}\label{lemma:glemma}
    Assume $\mathcal{X}$, $\mathcal{Y}$ are two disjoint set in $\Omega$ and $\mathrm{diam}(\Omega)=D$. Let $\delta>0$ be any positive constant, then if 
    \begin{equation}\label{equ:r}
        r > \max\left\{ \log_2\left( \frac{e^{2\varepsilon^2D^2}}{\delta} \right)-1, 12\varepsilon^2D^2-1  \right\}
    \end{equation}
     we have
    \begin{equation}\label{equ:alphabeta}
        |{e^{ - {\varepsilon ^2}{{(x - y)}^2}}}-\sum_{n=0}^r \alpha_n(x-\bar x)\beta_n(y-\bar x)|<\delta 
    \end{equation}
    for any $\bar x\in \mathcal{X}$. 
\end{lemma}
\begin{proof}
    See \cref{sect:proof}.
\end{proof}

\begin{remark}
    In practice, the estimate \cref{equ:r} is quite conservative and smaller $r$ can actually work very well. However, we need to point out that as the dimensionality increases, such method might suffer from the curse of dimensionality: if we use fix $r=5$ per dimension, the constructed low rank matrix has rank $5$ in 1D, $25$ in 2D, and $125$ in 3D.
\end{remark}

\begin{remark}
	The method proposed above, i.e., where we need to find a low-rank expansion of the kernel function, is by no means the only method to construct a $\mathcal{H}$-matrix. Other methods such as SVD decomposition~\cite{bebendorf2008hierarchical}, ACA~\cite{zhao2005adaptive}, Blackbox FMM~\cite{fong2009black}, hierarchical interpolative factorization~\cite{l2016hierarchical}, etc. In the numerical experiments, we implemented several methods and use appropriate methods for different problems. 
\end{remark}
By using the $\mathcal{H}$-matrix, the storage complexity is reduced to $\mathcal{O}(N)$ which is demonstrated in \cref{fig:construction}. The construction time is also reduced to $\mathcal{O}(N)$ compared to $\mathcal{O}(N^2)$ for full matrices.

In 2D, the Merton jump diffusion model read
\begin{equation}
    \nu(\bx) = \exp(-\varepsilon^2 \|\bx\|^2)
\end{equation}
with the kernel function
\begin{equation}
    k(\bx, \by) = \nu(\bx-\by) = \exp(-\varepsilon^2 \|\bx-\by\|^2)
\end{equation}

Let $\bx\in\mathcal{X}$, $\by\in\mathcal{Y}$ and $\mathcal{X}\cap \mathcal{Y} = \emptyset$, and assume that $\bar x\in \mathcal{X}$, 
\begin{equation}
    t_1 = \bx_1-\bar x_1\quad t_2 = \bx_2-\bar x_2\quad s_1 = \by_1-\bar y_1\quad s_2 = \by_2-\bar y_2\quad
\end{equation}
we have
\begin{equation}
    k(\bx,\by) = \sum\limits_{m,n = 0}^\infty  {\left[ {\frac{{{{(2{\varepsilon ^2})}^{m + n}}}}{{m!n!}}s_1^mt_1^n\exp \left( { - {\varepsilon ^2}(t_1^2 + s_1^2)} \right)} \right]} \left[ {s_2^mt_2^n\exp \left( { - {\varepsilon ^2}(t_2^2 + s_2^2)} \right)} \right]
\end{equation}

Let 
\begin{align}
    \alpha_{m,n} &= {\frac{{{{(2{\varepsilon ^2})}^{m + n}}}}{{m!n!}}s_1^mt_1^n\exp \left( { - {\varepsilon ^2}(t_1^2 + s_1^2)} \right)}\\
    \beta_{m,n} &= {s_2^mt_2^n\exp \left( { - {\varepsilon ^2}(t_2^2 + s_2^2)} \right)}
\end{align}
Similar to \cref{equ:alphabeta}, we can approximate the kernel using low rank summation
\begin{equation}
    k(\bx,\by) \approx \sum_{m,n=0}^{r} \alpha_{m,n}(s_1,t_1)\beta_{m,n}(s_2,t_2)
\end{equation}

Using the storage strategy in \Cref{equ:cons}, we can construct the $\mathcal{H}$-matrix directly. \Cref{fig:construction} shows the construction time as well as storage consumption. Notably, we compare the construction time of the $\mathcal{H}$-matrix with that of the dense matrix. We can see that the construction of $\mathcal{H}$-matrix is quite efficient, both in terms of storage consumption and time consumption: they both achieve an approximately linear asymptotic rate with respect to the problem size $N$. 
\begin{figure}[htpb]
\centering
\scalebox{0.9}{
\begin{tikzpicture}

\definecolor{color0}{rgb}{0.12156862745098,0.466666666666667,0.705882352941177}
\definecolor{color1}{rgb}{1,0.498039215686275,0.0549019607843137}
\definecolor{color2}{rgb}{0.172549019607843,0.627450980392157,0.172549019607843}
\definecolor{color3}{rgb}{0.83921568627451,0.152941176470588,0.156862745098039}

\begin{axis}[
legend cell align={left},
legend entries={{H-matrix},{Direct},{$O(N^2)$},{$O(N)$}},
legend style={at={(0.03,0.97)}, anchor=north west, draw=white!80.0!black},
tick align=outside,
tick pos=left,
title={Memory Used During Construction},
x grid style={white!69.01960784313725!black},
xlabel={Matrix Size},
xmin=724.077343935024, xmax=1482910.40037893,
xmode=log,
y grid style={white!69.01960784313725!black},
ylabel={Bytes},
ymin=2770112.72887551, ymax=337080041946.128,
ymode=log
]
\addlegendimage{no markers, color0}
\addlegendimage{no markers, color1}
\addlegendimage{no markers, color2}
\addlegendimage{no markers, color3}
\addplot [semithick, color0, mark=*, mark size=3, mark options={solid}]
table [row sep=\\]{%
1024	4716768 \\
2048	11744336 \\
4096	27932048 \\
8192	64503792 \\
16384	145966480 \\
32768	325468368 \\
65536	717564400 \\
131072	1567878416 \\
262144	3400688688 \\
524288	7330612272 \\
1048576	15719002736 \\
};
\addplot [semithick, color1, mark=*, mark size=3, mark options={solid}]
table [row sep=\\]{%
1024	8388704 \\
2048	33554528 \\
4096	134217824 \\
8192	536871008 \\
16384	2147483744 \\
32768	8589934688 \\
65536	34359738464 \\
};
\addplot [semithick, color2, dashed]
table [row sep=\\]{%
1024	48331026.4519679 \\
2048	193324105.807872 \\
4096	773296423.231486 \\
8192	3093185692.92594 \\
16384	12372742771.7038 \\
32768	49490971086.8151 \\
65536	197963884347.26 \\
};
\addplot [semithick, color3, dashed]
table [row sep=\\]{%
1024	4719826.80194999 \\
2048	9439653.60389999 \\
4096	18879307.2078 \\
8192	37758614.4155999 \\
16384	75517228.8311999 \\
32768	151034457.6624 \\
65536	302068915.324799 \\
131072	604137830.649599 \\
262144	1208275661.2992 \\
524288	2416551322.5984 \\
1048576	4833102645.19679 \\
};
\end{axis}

\end{tikzpicture}}~
\scalebox{0.9}{
\begin{tikzpicture}

\definecolor{color0}{rgb}{0.12156862745098,0.466666666666667,0.705882352941177}
\definecolor{color1}{rgb}{1,0.498039215686275,0.0549019607843137}
\definecolor{color2}{rgb}{0.172549019607843,0.627450980392157,0.172549019607843}
\definecolor{color3}{rgb}{0.83921568627451,0.152941176470588,0.156862745098039}

\begin{axis}[
legend cell align={left},
legend entries={{H-matrix},{Direct},{$O(N^2)$},{$O(N)$}},
legend style={draw=white!80.0!black},
tick align=outside,
tick pos=left,
title={Construction Time},
x grid style={white!69.01960784313725!black},
xlabel={Matrix Size},
xmin=724.077343935024, xmax=1482910.40037893,
xmode=log,
y grid style={white!69.01960784313725!black},
ylabel={Time (sec)},
ymin=0.00220187706027951, ymax=42434.4875567995,
ymode=log
]
\addlegendimage{no markers, color0}
\addlegendimage{no markers, color1}
\addlegendimage{no markers, color2}
\addlegendimage{no markers, color3}
\addplot [semithick, color0, mark=*, mark size=3, mark options={solid}]
table [row sep=\\]{%
1024	0.016117583 \\
2048	0.046379019 \\
4096	0.127069826 \\
8192	0.292432849 \\
16384	0.860845945 \\
32768	2.081860078 \\
65536	4.223938207 \\
131072	8.916075953 \\
262144	20.242718506 \\
524288	39.780334734 \\
1048576	81.861753704 \\
};
\addplot [semithick, color1, mark=*, mark size=3, mark options={solid}]
table [row sep=\\]{%
1024	0.028535442 \\
2048	0.132978801 \\
4096	1.452989489 \\
8192	3.543617585 \\
16384	35.091152046 \\
32768	159.240163905 \\
65536	642.000945215 \\
};
\addplot [semithick, color2, dashed]
table [row sep=\\]{%
1024	4.83310264519679 \\
2048	19.3324105807872 \\
4096	77.3296423231486 \\
8192	309.318569292594 \\
16384	1237.27427717038 \\
32768	4949.09710868151 \\
65536	19796.3884347261 \\
};
\addplot [semithick, color3, dashed]
table [row sep=\\]{%
1024	0.00471982680194999 \\
2048	0.00943965360389999 \\
4096	0.0188793072078 \\
8192	0.0377586144155999 \\
16384	0.0755172288311999 \\
32768	0.1510344576624 \\
65536	0.302068915324799 \\
131072	0.604137830649599 \\
262144	1.2082756612992 \\
524288	2.4165513225984 \\
1048576	4.83310264519679 \\
};
\end{axis}

\end{tikzpicture}}
\caption{The construction time and the storage consumption of $\mathcal{H}$-matrix. We compare the construction time of the $\mathcal{H}$-matrix with that of the dense matrix. We can see that the construction of $\mathcal{H}$-matrix is quite efficient, both in terms of storage consumption and time consumption: they both achieves an approximately linear asymptotic rate with respect to the problem size $N$.}
\label{fig:construction}
\end{figure}
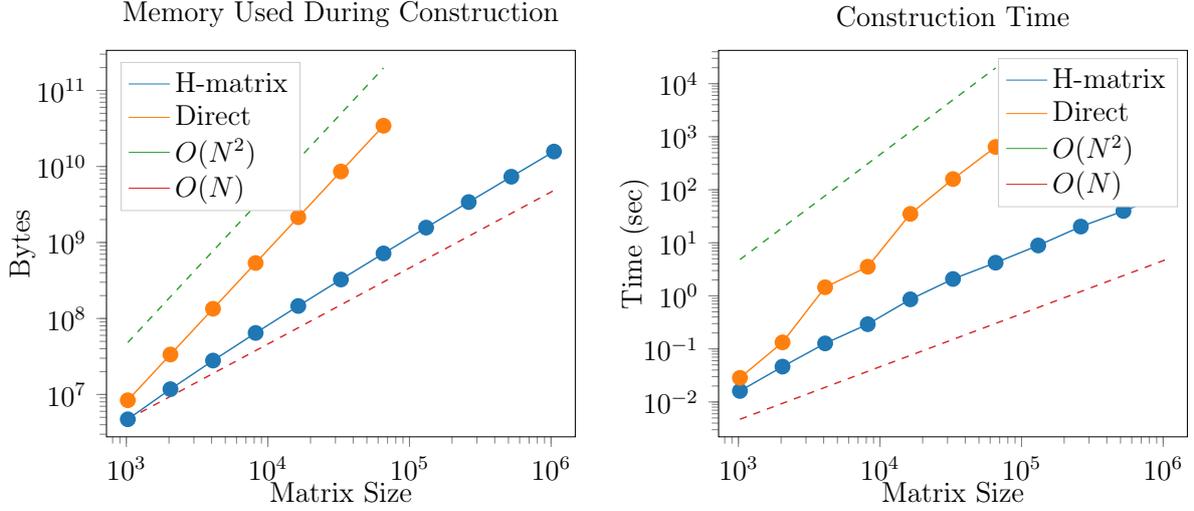

\subsection{Error Analysis}

\subsubsection{Stability}
We carry out the stability analysis using the Fourier transform pair~\cite{isaacson2012analysis}
\begin{align}
    u_j^n =& \frac{1}{2\pi h}\int_{-\pi}^\pi \hat u^n(x) \exp(\ii j x)dx\\
    \hat u_j^n =&  h \sum_{-\infty}^\infty u_j^n \exp(-\ii jx)dx
\end{align}

For simplicity, we assume that the spatial domain is not truncated, i.e., $\mathcal{I}=\mathbb{Z}$; another choice is to assume that $\nu(ih)=0$ for $i\not\in \mathcal{I}$. We have the following lemma
\begin{lemma}
    Let 
    \begin{equation}
        \eta_h(\theta) = \sum_{j=-\infty}^\infty (e^{\ii jh\theta}-1)\nu_jh
    \end{equation}
    be well defined for all $\theta\in\RR$, then we have
    \begin{equation}
        \delta_L e^{\ii \theta x} = \eta_h(\theta) e^{\ii \theta x}
    \end{equation}
    where $\nu_j = \nu(jh)$. In particular, if we split $\nu_j$ into odd part and even part
    \begin{equation}
    	\nu_j^e = \frac{\nu_j + \nu_{-j}}{2}\quad \nu_j^o = \frac{\nu_j - \nu_{-j}}{2}
    \end{equation}
    we have 
    \begin{align}
    	\eta_h(\theta) &=  \eta_h^e(\theta) + \ii\eta_h^o(\theta)\\
    	 \eta_h^e(\theta) &= -2\sum_{j=-\infty}^\infty \sin^2\left(\frac{j\theta h}{2} \right) \nu_j^eh\leq 0\\
    	 \eta_h^o(\theta) &= \sum_{j=-\infty}^\infty \sin(jh\theta) \nu_j^oh\geq 0
    \end{align}
\end{lemma}
\begin{proof}
    By definition, we have
    \[{\delta _L}{e^{\ii\theta x}} = \int_{ - \infty }^\infty  {\left( {{e^{\ii(x + y)\theta }} - {e^{\ii x\theta }}} \right)\nu (y)dy = {\eta _h}(\xi )} {e^{\ii\theta x}}\]
   In addition, direct computation yields
   \begin{align}
   	\eta _h^e(\theta ) =& \sum\limits_{j =  - \infty }^\infty  {({e^{iijh\theta }} - 1)\nu _j^eh} \\
   	=&   - 2\sum\limits_{j =  - \infty }^\infty  {{{\sin }^2}\left( {\frac{{j\theta h}}{2}} \right)\nu _j^eh}  = - 2\sum\limits_{j =  - \infty }^\infty  { {{\sin }^2}\left( {\frac{{j\theta h}}{2}} \right)\nu _j^{}h} 
   \end{align}
   the same is true for $\eta_h^o(\theta)$
\end{proof}

\begin{remark}
	In the case $\nu(y)$ is symmetric, $\nu_j^o=0$, and therefore we have
    \begin{equation}
        \eta_h(\theta) = -2\sum_{j=-\infty}^\infty \sin^2\left(\frac{j\theta h}{2} \right) \nu_jh\leq 0
    \end{equation}
\end{remark}

The Fourier transform of the numerical scheme gives
\begin{equation}\label{equ:von}
    \hat u_i^{n+1} = \frac{{1 - a\frac{{\Delta t}}{{{h^2}}}{{\sin }^2}\frac{\theta }{2} + \frac{{b\Delta t}}{{2h}}\ii\sin \theta  + \frac{{c\Delta t}}{2}+ \frac{{\Delta t\eta_h (\theta )}}{2}}}{{1 + a\frac{{\Delta t}}{{{h^2}}}{{\sin }^2}\frac{\theta }{2} - \frac{{b\Delta t}}{{2h}}\ii\sin \theta  - \frac{{c\Delta t}}{2} - \frac{{\Delta t\eta_h (\theta )}}{2}}} \hat u_i^n
\end{equation}

Note we have
\begin{multline}
{\left| {1 - a\frac{{\Delta t}}{{{h^2}}}{{\sin }^2}\frac{\theta }{2} + \frac{{b\Delta t}}{{2h}}\ii\sin \theta  + \frac{{c\Delta t}}{2} + \frac{{\Delta t\eta_h^e (\theta )}}{2}} \right|^2} = \\
{\left| {1 - a\frac{{\Delta t}}{{{h^2}}}{{\sin }^2}\frac{\theta }{2} + \frac{{c\Delta t}}{2} + \frac{{\Delta t\eta_h^e (\theta )}}{2}} \right|^2} + {\left| {\frac{{b\Delta t}}{{2h}}\sin \theta } +\frac{\Delta t \eta^o_h(\theta)}{2}\right|^2}
\end{multline}
\begin{multline}
\left| {1 + a\frac{{\Delta t}}{{{h^2}}}{{\sin }^2}\frac{\theta }{2} - \frac{{b\Delta t}}{{2h}}\ii\sin \theta  - \frac{{c\Delta t}}{2} - \frac{{\Delta t\eta_h^e (\theta )}}{2}} \right| = \\
{\left| {1 + a\frac{{\Delta t}}{{{h^2}}}{{\sin }^2}\frac{\theta }{2} - \frac{{c\Delta t}}{2} - \frac{{\Delta t\eta_h^e (\theta )}}{2}} \right|^2} + {\left| {\frac{{b\Delta t}}{{2h}}\sin \theta } +\frac{\Delta t \eta^o_h(\theta)}{2} \right|^2}
\end{multline}

Since we have $a\geq 0$, $c\leq 0$, $\eta_h(\theta)\leq  0$, we always have
\begin{equation}
    \left| {1 + a\frac{{\Delta t}}{{{h^2}}}{{\sin }^2}\frac{\theta }{2} - \frac{{c\Delta t}}{2} - \frac{{\Delta t\eta^e_h (\theta )}}{2}} \right| \ge \left| {1 - a\frac{{\Delta t}}{{{h^2}}}{{\sin }^2}\frac{\theta }{2} + \frac{{c\Delta t}}{2} + \frac{{\Delta t\eta^e_h (\theta )}}{2}} \right|
\end{equation}

Therefore, the model of the ratio in \cref{equ:von} is always no greater than 1. Thus all the wave modes $e^{\ii \theta x}$ will not grow in magnitude if we carry out the Crank-Nicolson scheme. To summarize, we have proved
\begin{lemma}[Stability]
Assume that $\mathcal{I}=\mathbb{Z}$. Then the Crank Nicolson scheme \cref{equ:scheme} is unconditionally stable. 
\end{lemma} 

\begin{remark}
	For simplicity, we have assumed that the domain is not truncated, i.e. $\mathcal{I}=\mathbb{Z}$; in practice, we cannot have infinite number of unknowns $u^{n+1}_j$ and need to impose artificial boundary conditions. This truncation can have undesired impact on the accuracy in the numerical scheme and therefore reduce the convergence order, especially when the L\'evy measure has a heavy tail. See remarks in \Cref{sect:extension} for more details. 
\end{remark}

\subsubsection{Consistency}
In consideration of \cref{equ:small}, we assume that 
\begin{equation}\label{equ:assumption}
    \nu(y) = 0 \quad y>B_r \mbox{ or } y<B_l
\end{equation}

The consistency is a direct result of the Crank Nicolson scheme. Note that \cref{equ:dl} is the trapezoidal discretization of the nonlocal operator, we have
\begin{equation}
    \int_{B_l}^{B_r} (u(x_j+y)-u(x_j))\nu(y)dy = (\delta_L u)_j + \mathcal{O}(h^2)
\end{equation}
and therefore
\begin{equation}\label{equ:deltal}
  \begin{aligned}
    &\frac{{({\delta _L}u)_j^n + ({\delta _L}u)_j^{n + 1}}}{2} \\
    =& \int_{{B_l}}^{{B_r}} {\left( {\frac{{u({x_{i + j}} + y,{t_n}) + u({x_{i + j}} + y,{t_{n + 1}})}}{2} - \frac{{u({x_i} + y,{t_n}) + u({x_i} + y,{t_{n + 1}})}}{2}} \right)} \nu (y)dy + \mathcal{O}( {h^2})\\
     =& \int_{{B_l}}^{{B_r}} {\left( {u\left( {{x_{i + j}} + y,{t_{n + \frac{1}{2}}}} \right) - u\left( {{x_i} + y,{t_{n + \frac{1}{2}}}} \right)} \right)} \nu (y)dy + \mathcal{O}(\Delta {t^2} + {h^2})\\
      =& {\cal L}u\left( {{x_i},{t_{n + \frac{1}{2}}}} \right) + \mathcal{O}(\Delta {t^2} + {h^2})
\end{aligned}
\end{equation}
It is a standard result that~\cite{giles2005convergence}
\begin{multline}\label{equ:deltar}
    \frac{{(a\delta _x^2 + b{\delta _{2x}} + c)u(x_i,t_n) + (a\delta _x^2 + b{\delta _{2x}} + c)u(x_i,t_{n+1})}}{2} \\= a{u_{xx}}\left( {{x_i},{t_{n + \frac{1}{2}}}} \right) + b{u_x}\left( {{x_i},{t_{n + \frac{1}{2}}}} \right) + cu\left( {{x_i},{t_{n + \frac{1}{2}}}} \right)+\mathcal{O}(h^2+\Delta t^2)
\end{multline}
and that
\[\frac{{u({x_i},{t_{n + 1}}) - u({x_i},{t_n})}}{{\Delta t}} = \mathcal{O}(\Delta t^2)\]
therefore combining \cref{equ:deltal,equ:deltar} we have
\begin{lemma}[Consistency]\label{lemma:consistency}
    Assume \cref{equ:assumption} holds. Then the truncation error for the numerical scheme \cref{equ:scheme}  
    \begin{multline}
        T^n_i := \frac{{u({x_i},{t_{n + 1}}) - u({x_i},{t_n})}}{{\Delta t}}\\
         - \frac{{(a\delta _x^2 + b{\delta _{2x}} + c + {\delta _L})u({x_i},{t_n}) + (a\delta _x^2 + b{\delta _{2x}} + c + {\delta _L})u({x_i},{t_{n + 1}})}}{2}
    \end{multline}
    satisfies
    \begin{equation}
        T^n_i = \mathcal{O}(\Delta t^2 + h^2)
    \end{equation}
    
\end{lemma}

\subsubsection{Convergence}

Finally, we are in a position to prove the convergence of the numerical scheme \cref{equ:scheme}. 

\begin{theorem}
    Assume that $\nu(y)\in C(\RR)$ and the condition in \cref{lemma:consistency} is satisfied. Let $u_i^n$ be the numerical solution at $x_i$ and time $t_n$, and $u(x, t)$ be the exact solution. Then the numerical scheme \cref{equ:scheme} is unconditionally stable and 
    \begin{equation}
        |u(x_i,t_n)-u_i^n| = \mathcal{O}(\Delta t^2 + h^2) \quad \Delta t \rightarrow 0, h\rightarrow 0
    \end{equation}
\end{theorem}
\begin{proof}
    The theorem is a direct result that the stability and consistency imply convergence~\cite{giles2005convergence}. 
\end{proof}

\lib{Singular and/or Slow Decaying L\'evy Measure: the Fractional Laplacian}\label{sect:extension}

We now consider the general case where $\nu(y)$ is singular at $y=0$ or has a heavy tail instead of the assumption $\nu(y)<\infty$ and $\nu(y)$ is semi-heavy in the previous sections. We will only state the algorithm in 1D, but point out that it can be directly generalized to higher dimensions and demonstrate its validity in the numerical examples. 

One such example is the fractional Laplacian where the L\'evy measure is 
\begin{equation}
    \nu(y) = \frac{c_{1,s}}{|y|^{1+2s}} \qquad c_{1,s} = \frac{2^{2s}\Gamma\left( \frac{1+2s}{2} \right)}{\pi^{1/2}\left| \Gamma\left(-s\right) \right|}
\end{equation}
where $s\in (0,1)$. Note in this case, $-\int_{\RR}(u(x+y)-u(x))\nu(y)dy$ must be understood in the principal value integration. The corresponding stochastic process associated with the fractional Laplacian is the $\alpha$-stable process. 

Consider the general singular integral operator
\begin{equation}\label{equ:Ix}
    I(x) = \int_{\RR}(u(x+y)-u(x)-\rho(y)u'(x)y)\nu(y)dy
\end{equation}
where $\rho(y)u'(x)$ is a drift term to remove small activity from the jumps. $\rho(y)$ is a radial symmetric window function, satisfying
\begin{equation}\label{equ:rho_condition}
\begin{cases}
    1-\rho(y)\sim \mathcal{O}(y^4)& y\rightarrow 0\\
    \rho(y)=0 & |y|\geq r
\end{cases}
\end{equation}
where $r>0$ is a positive number. 

As a reminder, we require $\nu(y)$ to satisfy the following conditions
\begin{equation}\label{equ:ny}
    \int_{-r}^ry^2\nu(y)dy <\infty, \quad \int_{|y|\geq r} \nu(y)dy<\infty
\end{equation} 
where $r>0$ is a constant. 

The choice of $\rho(y)$ doesn't matter. In fact, if $\tilde\rho(x)$ is another window function that satisfies \cref{equ:rho_condition}, we have
\begin{align}
    &\int_{\RR}(u(x+y)-u(x)-\tilde\rho(y)u'(x)y)\nu(y)dy \\
    =& \int_{\RR}(u(x+y)-u(x)-\rho(y)u'(x)y)\nu(y)dy + u'(x)\int_{\RR}(\rho(y)-\tilde\rho(y)y)\nu(y)dy
\end{align}
we can add the second term to the drift term in the model. 

The first condition in \cref{equ:rho_condition} is designed to take into consideration of the heavy tail case, where $\nu(y)$ can decay like $\mathcal{O}(1/|y|^{1+2s})$ for some $s\in (0,1)$. For example, in the special case     $\nu(y) = c_{1,s}\frac{1}{|y|^{1+2s}}$, $\int_{\RR}(u(x+y)-u(x))\nu(y)dy$ is not well defined but only in the principle value integration, and we have
\begin{equation}
    \mathrm{p.v.}\int_{\RR}(u(x+y)-u(x))\nu(y)dy = \int_{\RR}(u(x+y)-u(x)-\rho(y)u'(x)y)\nu(y)dy
\end{equation}
for any valid window function $\rho(y)$ thanks to the cancellation of the drift term due to symmetry. 

Although $\int_{\RR}(u(x+y)-u(x)-\rho(y)u'(x)y)\nu(y)dy$ is well-defined in this case, the integrand will behave like 
\begin{equation}
    (u(x+y)-u(x)-\rho(y)u'(x)y)\nu(y) = \mathcal{O}\left(\frac{1}{|y|^{2s-1}}\right)
\end{equation}
in the case $s\rightarrow 0+$, we will have numerical difficulty if a direct numerical integration is applied, especially for $s<\frac{1}{2}$. In the following, we will propose a numerical discretization for \cref{equ:Ix} targeting at the most numerical challenging case described above 
\begin{equation}\label{equ:ny2}
    \nu(y) = \frac{n_0(y)}{|y|^{1+2s}}
\end{equation} 
where $n_0(y)$ is a bounded continuous function. 

We make two assumptions on $u(x)$
\begin{itemize}
    \item $u\in C(\RR)$ 
    \item Local smoothness. $u\in C^4([x-\delta, x+\delta])$ for some $\delta>0$, i.e., $u$ has fourth order derivative near the location where we want to evaluate $I(x)$.
    \item Far field asymptotic limit. Assume $L_W>r$. The far field contribution
    \begin{equation}\label{equ:fxy}
        f_x^{L_W}(y) = \int_{|y|>L_W}u(x+y)\nu(y)dy
    \end{equation}
    is well defined. In the case $\frac{u(x+y)}{u(y)}\rightarrow f(y)$, this term can be approximated by $\int_{|y|>L_W} f(y)dy$
\end{itemize}


The strategy is the singularity subtraction, which is one of the standard method in treating singular integrals in BEM~\cite{minden2018simple,jarvenpaa2006singularity,wilde1998direct,hanninen2006singularity,anselone1981singularity}. We subtract a local diffusion term from \cref{equ:Ix}
\begin{multline}\label{equ:evalIx1d}
    I(x) = \int_{\RR}(u(x+y)-u(x)-\rho(y)u'(x)y- \frac{1}{2}\rho(y)u''(x)y^2 )\nu(y)dy \\
    + \frac{1}{2}u''(x)\int_\RR \rho(y)\nu(y)y^2 dy
\end{multline}
We can immediately split the first integral into two parts
\begin{equation}
    I_1(x) = \int_{|y|\leq L_W}(u(x+y)-u(x)-\rho(y)u'(x)y- \frac{1}{2}\rho(y)u''(x)y^2 )\nu(y)dy
\end{equation}
and
\begin{align}
    I_2(x) =& \int_{|y|> L_W}(u(x+y)-u(x)-\rho(y)u'(x)y- \frac{1}{2}\rho(y)u''(x)y^2 )\nu(y)dy\\
     = &\int_{|x|> L_W}(u(x+y)-u(x) )\nu(y)dy = f_x^{L_W} - u(x)\int_{|y|> L_W}\nu(y)dy
\end{align}

By Taylor expansion, it is easy to see
\begin{equation}
    u(x+y)-u(x)-\rho(y)u'(x)y- \frac{1}{2}\rho(y)u''(x)y^2 =\mathcal{O}(|y|^3)
\end{equation}
and therefore the integrand of $I_1(x)$ will behave like $\mathcal{O}(|y|^{2-2s})$ near the origin. Since $2-2s\geq 0$, the integrand becomes continuous near the origin. Thus $I_1(x)$ is well defined.   

As $y\rightarrow \infty$, the term terms in $I_2(x)$ are both well defined according to the assumptions \cref{equ:fxy,equ:ny}. 

The second term in \cref{equ:evalIx1d} 
\begin{equation}
    I_3(x) = \frac{1}{2}u''(x)\int_{|y|\leq r}\nu(y)y^2 dy
\end{equation}
is a local diffusion term and the coefficient is well defined according to \cref{equ:ny}.

We now focus on the numerical discretization of $I_1(x)$, $I_2(x)$ and $I_3(x)$. We divide the mask window into $2N$ uniform subintervals and consider the grid $\{ih:i\in\mathbb{Z} \}$, where $h = L_W/N$. We denote $u_i = u(ih)$.

Since the integrand in $I_1(x)$ is continuous, we can use a simple trapezoidal quadrature rule to approximate the integral. Assume the quadrature weights are $w_j$ given by $w_{-N}=w_N = \frac{h}{2}$ and $w_j=h$, $j=-N+1, -N+2, \ldots, N-2, N-1$.
\begin{equation}\label{equ:I1approx}
  \begin{aligned}
    I_1(x_i) &\approx {\sum\limits_{j =  - N}^N}' {{u_{i + j}}\nu(jh){w_j}}  - {u_i}{\sum\limits_{j =  - N}^N}' {\nu(jh){w_j}} \\
    & - \frac{{{u_{i + 1}} - {u_{i - 1}}}}{h}{\sum\limits_{j =  - N}^N}' {\rho (jh)\nu(jh)jh}  - \frac{{{u_{i + 1}} + {u_{i - 1}} - 2{u_i}}}{{2{h^2}}}{\sum\limits_{j =  - N}^N}' {\rho (jh){{(jh)}^2}\nu(jh)} 
\end{aligned}
\end{equation}
where ${\sum\limits_{j =  - N}^N}'$ denotes the summation excluding $j=0$. 

For $I_2(x)$, $f_x^{L_W}$ is either provided as an input or computed using a numerical quadrature and so is $\int_{|x|>L_W} \nu(y)dy$. We will see how these terms are obtained in the examples below. 
\begin{equation}\label{equ:I2approx}
    I_2(x_i)\approx f_{x_i}^{L_W} - u(x_i)\int_{|y|> L_W}\nu(y)dy
\end{equation}
For $I_3$, a central difference scheme is applied to the second order derivative term. 
\begin{equation}\label{equ:I3approx}
    I_3(x_i)\approx \frac{{{u_{i + 1}} + {u_{i - 1}} - 2{u_i}}}{{2{h^2}}}\int_{|y| \leqslant r} n (y){y^2}dy
\end{equation}
and the integral can either be computed analytically or numerically.

In practice, we want to compute $I(x)$ for $x\in [-L,L]$, according to \cref{equ:I1approx}, we need to know the values of $u(x)$ on $[-L-L_W,L+L_W]$ and its corresponding far-field interactions. \Cref{fig:fig31} visualizes the relationship. To compute $I(x_i)$, we need to compute the near field interaction and local interaction using values of $u(x)$ from the green area. The values of $u(x)$ are provided in the green and red area for computing $I(x)$, $x\in [-L, L]$.

\begin{figure}[H] 
\centering
\includegraphics[width=0.8\textwidth,keepaspectratio]{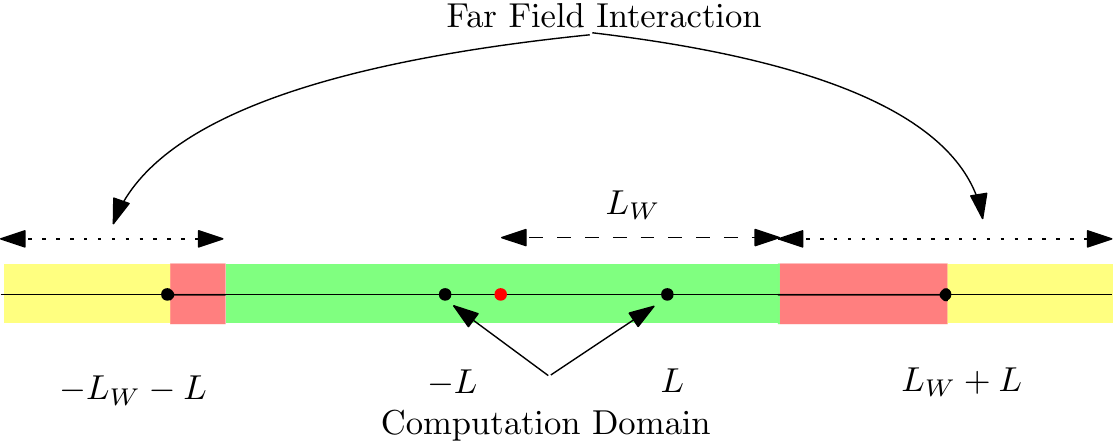}
\caption{To compute $I(x_i)$, we need to compute the near field interaction and local interaction using values of $u(x)$ from the green area. The values of $u(x)$ are provided in the green and red area for computing $I(x)$, $x\in [-L, L]$.}
\label{fig:fig31}
\end{figure}

In sum, we have the formula
\begin{equation}
    (\delta_L u)_i = I_1(x_i) + I_2(x_i) + I_3(x_i) = \bw_i^T \bu + \mathbf{f}_i
\end{equation}
for some vector $\bw_i\in \RR^{|\mathcal{I}|}$, $\mathbf{f}_i\in \RR^{|\mathcal{I}|}$ and 
\begin{equation}
    \bu = (u_i)_{i\in \mathcal{I}}
\end{equation}

Although we have used a different formula for the evaluation of the integral, we should soon realize that in the far-away off-diagonal parts, the entries are still $\nu_j h$~(except on the boundary), which the $\mathcal{H}$ matrix construction routine can still work. 

\begin{remark}
	In this section, we presents an approach to evaluate the singular integral operator  \cref{equ:Ix} where $\nu(y)$ can have singularity at $y=0$ and a heavy tail. A particular example is the fractional Laplacian. However, in practice, it is not easy to obtain $f_x^{L_W}$, especially for higher dimensions. One way to circumvent this difficulty is to enforce $u(\bx)=0$ outside a bounded domain. For example, $u(x)=0$, $\forall x\in [-L_W,L_W]^2$ in this case and thus $f_x^{L_W}=0$. 
	
	We must be cautious about the simple truncation. It was shown~\cite{ros2014dirichlet} that if $u$ is a solution of $(-\Delta)^s u = g$ in $\Omega$, $u\equiv 0$ in $\RR^d\backslash \Omega$ for some $s\in (0,1)$ and $g\in L^\infty(\Omega)$, then $u$ is $C^s(\RR^d)$ and $u/\delta^s|_\Omega$ is $C^\alpha$ up to the boundary $\partial\Omega$ for some $\alpha\in (0,1)$, where $\delta(x)=\mathrm{dist}(x,\partial\Omega)$. This fact indicates that we will usually not expect ``optimal'' convergence of typical numerical schemes if we go for this simplicity. 
\end{remark}

\lib{Numerical Examples}

In this section, we carry out various numerical experiments with a focus on efficiency. The correctness is checked with either numerical results from the direct method or analytical solution. The algorithms are implemented using \texttt{julia-1.0.2} and run on a Ubuntu server with Intel(R) Xeon(R) CPU E7-8890 v3 @ 2.50GHz. 

Highly efficient $\mathcal{H}$-matrix is tricky to implement and depends on the choice of appropriate parameters based on specific kernels. We do not focus on tuning for the optimal parameters but focus on a general and straightforward implementation which can be easily adapted for other kernels. However, we mention that we can indeed improve the efficiency by devoting more effort for individual kernels, such as optimal parameter tuning, adapted rank strategy, and so on~\cite{coulier2017inverse,pouransari2017fast}. Our program only has two parameters $N_{\mathrm{block}}$ and $N_{\min}$ which is described in \Cref{equ:cons}, and in the numerical examples, we show that for large matrices, the efficiency is not sensitive to these parameters. We demonstrate that the general program can work very well compare to the baseline approach.

\subsection{Efficiency of $\mathcal{H}$-Matrix: 1D Case}

In this section, we show the efficiency of the arithmetic operations using the $\mathcal{H}$-matrix in 1D. In the experiment, the minimum block size is 64. The matrix sizes tested are $2^{10}, 2^{11}, \ldots, 2^{20}$. The maximum block size for $2^n\times 2^n$ matrices is $2^{n-2}\times 2^{n-2}$. The rank for off-diagonal approximation is $r=10$, which is quite accurate for the Gaussian kernel we considered. The authors observed that for $2^{17}\times 2^{17}$ matrices, the dense LU will throw \texttt{OutOfMemory} error and therefore for numerical experiments we stopped at $2^{17}\times 2^{17}$ for dense LU. Remarkably, we show that with the $\mathcal{H}$-matrix technique, we are able to LU factorize a one million by one million dense matrix with only 125 seconds without any explicit parallelism effort in \texttt{julia}. 

Consider the model problem 
\begin{equation}\label{equ:model1d}
	\begin{cases}
		u_t = \int_{\RR} (u(x+y)-u(x))e^{-5y^2} dy & (x,t)\in \RR\times (0,1]\\
		u(x,0) = e^{-50|x|^2} &   x\in \RR
	\end{cases}
\end{equation}

We divide the interval $[-1,1]$ into $2^n$ equal length intervals, $h=\frac{1}{2^{n-1}}$. For admissibility condition,  we use $\eta=1$. For low-rank blocks, the rank is fixed to be 10. In fact, the rank can be chosen adaptively; however, we observe that the fixed rank strategy is practical for our cases. 

A typical hierarchical matrix in 1D will have the skeleton shown in \cref{fig:1D}. Here we use a different color for each block. The green block denotes low-rank matrices while the yellow block denotes full matrices. The matrix is arranged into a hierarchical structure, from which $\mathcal{H}$-matrix got its name. 

\begin{figure}[H] 
\centering
\includegraphics[width=0.7\textwidth,keepaspectratio]{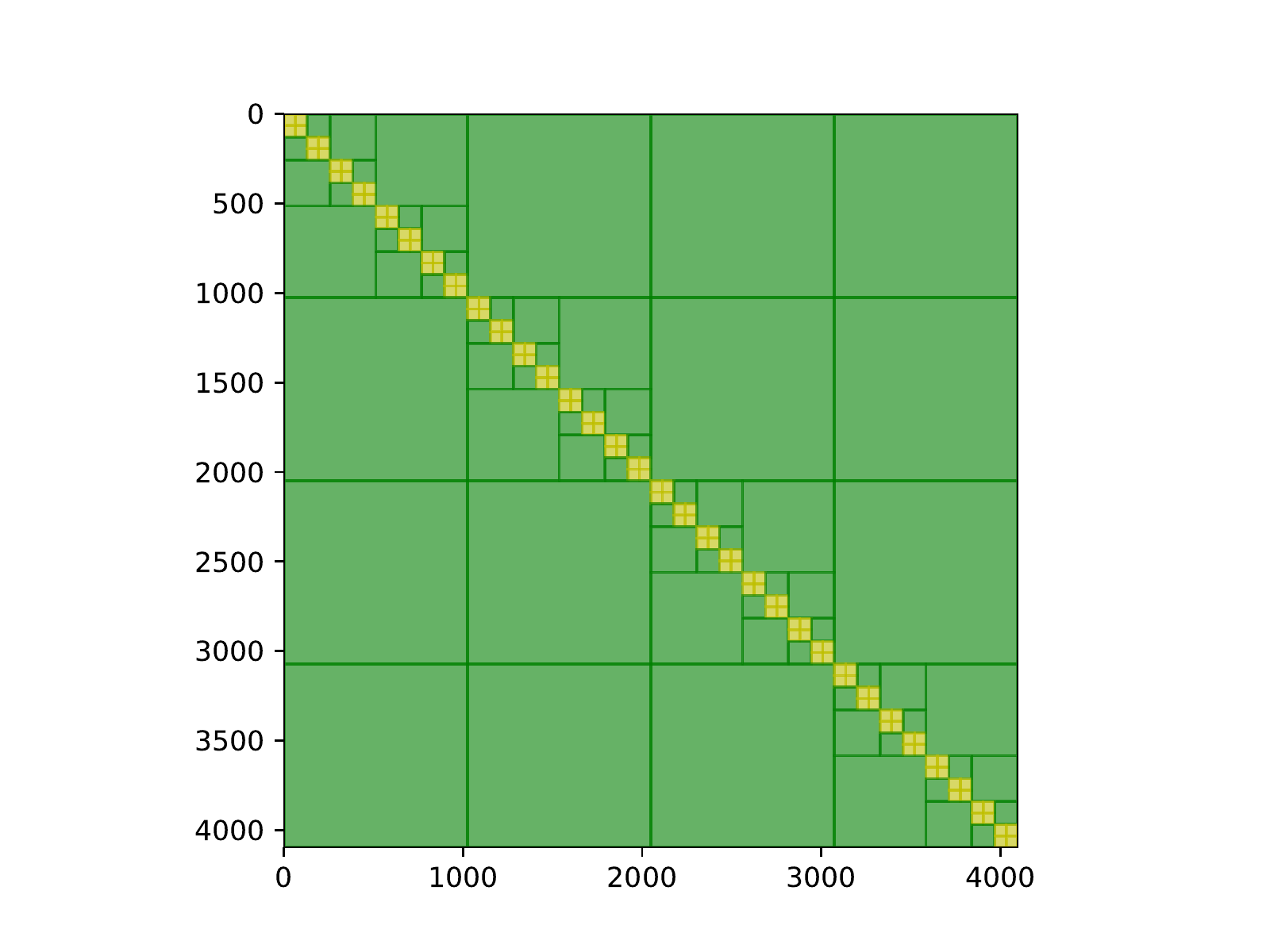}
\caption{1D $\mathcal{H}$-matrix. The green block denotes low-rank matrices while the yellow block denotes full matrices.}
\label{fig:1D}
\end{figure}

The key for maintaining optimal rates while the problem size $N$ becomes large is to control the total number of dense blocks. In principle, the number of dense blocks should grow linearly with problem size, which can be demonstrated by looking at the compression ratio or the number of total blocks~(full dense blocks as well as low-rank blocks).

\paragraph{Matrix Vector Multipliction}

The upper right plot in \cref{fig:g1d} shows the complexity of the matrix-vector multiplication for full matrices and $\mathcal{H}$-matrix. Compared to the dense matrix-vector multiplication, the $\mathcal{H}$-matrix structure lends us great speedup. This enables us to device highly efficient iterative solvers, such as preconditioned conjugate gradient method, which may require many matrix-vector productions during the iterations.

\paragraph{LU Decomposition}

We have already shown that the storage and construction complexity is $\mathcal{O}(N)$ in \cref{fig:construction}. In the lower left plot in \cref{fig:g1d} we also show that the LU decomposition is also much more efficient using the $\mathcal{H}$-matrices. We see that the $\mathcal{H}$-LU has better asymptotic complexity than the dense LU, which has complexity $\mathcal{O}(N^3)$. Note the $\mathcal{H}$-LU decomposition is carried out using high accuracy and can serve as a direct solver for linear systems. We need to point out that although the $\mathcal{H}$-LU tends to beat dense LU in terms of time consumption for large-scale problems, the constant in the asymptotic rate $\mathcal{O}(N)$ is still large, which is well-known in the literature.

\paragraph{Solve}

One crucial step for a successful implicit scheme is to solve the equation $A\bx=\by$. We can, of course, use matrix-free solvers such as PCG. However, in the case that $A$ is ill-conditioned, we may require a good preconditioner. Finding such a preconditioner is not an easy task, especially for the dense matrices, which is not covered by literature as comprehensively as that of sparse counterparts. $\mathcal{H}$-LU lends us a generic way to construct preconditioners or direct solvers. In both cases, we need to factorize $A$ as mentioned, and then solve $A\bx=\by$.

\paragraph{Solution to the Model Problem} 

We apply the $\mathcal{H}$-matrix technique developed in the paper to solve the model problem \cref{equ:model1d}. We first form a $\mathcal{H}$-matrix $H$, as an approximation to the stiffness matrix and LU factorize $H$ to obtain a factorized form $H_1$. $H_1^{-1}$ is then used as a preconditioner for solving the linear system with coefficients matrix $H$. For determine the accuracy of the numerical scheme, we solve the same model problem using an accurate numerical scheme and obtain a reference solution. First, we fix $N_T=100$ and $N=2^{10}$ and apply the Crank Nicolson scheme without $\mathcal{H}$-matrix approximation; later we solve the problem using $\mathcal{H}$-matrix approximation for $N_T=10$, $20$, $\ldots$, $50$, $N = 2^{10}$ and compare the solution at $t=1$ with the reference solution. Next, we fix $N_T=100$, $N=2^{15}$ and obtain a reference solution; we redo the computation with $\mathcal{H}$-matrix approximation with $N=2^8$, $2^9$, $\ldots$, $2^{13}$, $N_T=100$, and compare the solution at $t=1$ with the reference solution~(in this case we need to restrict the reference solution onto a coarser grid for comparison).

\Cref{fig:g1d_dtdx} shows the convergence plots as we increase $N_T$ and $N$. We see a second order convergence in time, which is consistent with our analysis. However, we only see a first order convergence. It is due to the  artificial truncation we have performed for tractable computation. For more details of the reduced convergence issue for nonlocal operators on the bounded domain, see remarks in \cref{sect:extension}.

\begin{figure}[htpb]
\centering
\scalebox{0.8}{
\begin{tikzpicture}

\definecolor{color0}{rgb}{0.12156862745098,0.466666666666667,0.705882352941177}
\definecolor{color1}{rgb}{1,0.498039215686275,0.0549019607843137}

\begin{axis}[
legend cell align={left},
legend entries={{$\mathcal{O}(N_T^{-2})$}},
legend style={draw=white!80.0!black},
tick align=outside,
tick pos=left,
x grid style={white!69.01960784313725!black},
xlabel={$N_T$},
xmin=9.22680834590588, xmax=54.1899193367184,
xmode=log,
y grid style={white!69.01960784313725!black},
ylabel={$||u-u_{ref}||_\infty$},
ymin=1.64796244443072e-05, ymax=0.00120869735091341,
ymode=log
]
\addlegendimage{no markers, color1}
\addplot [semithick, color0, mark=asterisk*, mark size=3, mark options={solid}, forget plot]
table [row sep=\\]{%
10	0.0006624088617109 \\
20	0.000160331204672987 \\
30	6.75272749546818e-05 \\
40	3.50586666398311e-05 \\
50	2.0032579160445e-05 \\
};
\addplot [semithick, color1, dashed]
table [row sep=\\]{%
10	0.000994324208098597 \\
11	0.000821755543883138 \\
12	0.000690502922290693 \\
13	0.000588357519584969 \\
14	0.00050730826943806 \\
15	0.000441921870266043 \\
16	0.000388407893788515 \\
17	0.000344056819411279 \\
18	0.000306890187684752 \\
19	0.000275436068725374 \\
20	0.00024858105202465 \\
21	0.000225470341972471 \\
22	0.000205438885970785 \\
23	0.000187962988298412 \\
24	0.000172625730572673 \\
25	0.000159091873295776 \\
26	0.000147089379896242 \\
27	0.000136395638971001 \\
28	0.000126827067359515 \\
29	0.000118231178133008 \\
30	0.000110480467566511 \\
31	0.000103467659531592 \\
32	9.71019734471287e-05 \\
33	9.1306171542571e-05 \\
34	8.60142048528198e-05 \\
35	8.11693231100896e-05 \\
36	7.6722546921188e-05 \\
37	7.26314249889407e-05 \\
38	6.88590171813434e-05 \\
39	6.53730577316632e-05 \\
40	6.21452630061624e-05 \\
41	5.91507559844495e-05 \\
42	5.63675854931178e-05 \\
43	5.37763227743969e-05 \\
44	5.13597214926962e-05 \\
45	4.91024300295604e-05 \\
46	4.6990747074603e-05 \\
47	4.50124132231144e-05 \\
48	4.31564326431683e-05 \\
49	4.14129199541274e-05 \\
50	3.97729683239439e-05 \\
};
\end{axis}

\end{tikzpicture}}~
\scalebox{0.8}{
\begin{tikzpicture}

\definecolor{color0}{rgb}{0.12156862745098,0.466666666666667,0.705882352941177}
\definecolor{color1}{rgb}{1,0.498039215686275,0.0549019607843137}

\begin{axis}[
legend cell align={left},
legend entries={{$\mathcal{O}(h)$}},
legend style={draw=white!80.0!black},
tick align=outside,
tick pos=left,
x grid style={white!69.01960784313725!black},
xlabel={$\frac{1}{h}$},
xmin=215.269482304951, xmax=9741.98468610228,
xmode=log,
y grid style={white!69.01960784313725!black},
ylabel={$||u-u_{ref}||_\infty$},
ymin=2.19897327155076e-07, ymax=0.0013450742891909,
ymode=log
]
\addlegendimage{no markers, color1}
\addplot [semithick, color0, mark=asterisk*, mark size=3, mark options={solid}, forget plot]
table [row sep=\\]{%
256	0.00090496428235376 \\
512	1.91334421458674e-05 \\
1024	9.71633609885914e-06 \\
2048	4.77446625544563e-06 \\
4096	2.24527093684208e-06 \\
8192	9.65974865369326e-07 \\
};
\addplot [semithick, color1, dashed]
table [row sep=\\]{%
256	1.04588699213203e-05 \\
512	5.22943496066016e-06 \\
1024	2.61471748033008e-06 \\
2048	1.30735874016504e-06 \\
4096	6.5367937008252e-07 \\
8192	3.2683968504126e-07 \\
};
\end{axis}

\end{tikzpicture}}
\caption{1D case \cref{equ:model1d}. \textit{Left}: Convergence as we increase $N_T$. We see a second order convergence in time, which is consistent with our analysis. \textit{Right}: Convergence as we increase $N$. Here we only see a first order convergence. It is due to the  artificial truncation we have performed for tractable computation.}
\label{fig:g1d_dtdx}
\end{figure}
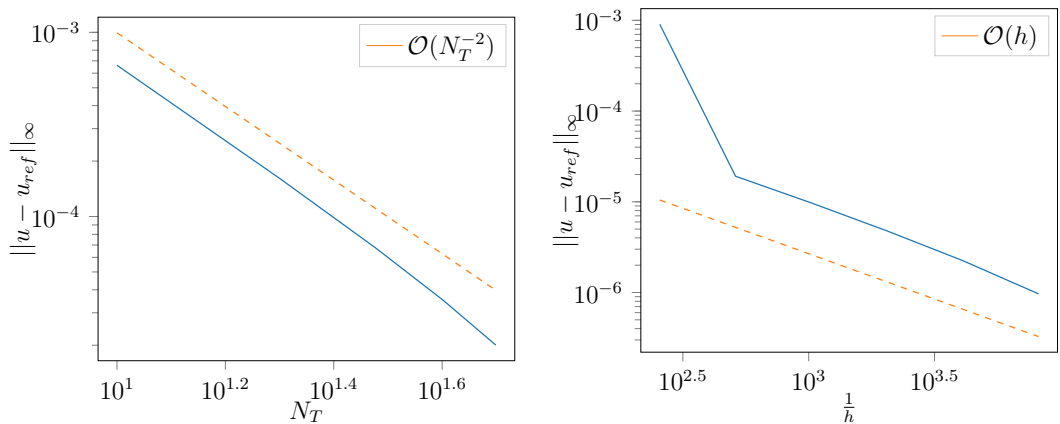

In this numerical experiment, we generate a random vector and record the solving time for both factorized $\mathcal{H}$-matrix and LU factorized dense matrix. The last plot in \cref{fig:g1d} compares the solving time for both the dense matrix and the $\mathcal{H}$-matrix. We see that the $\mathcal{H}$-matrix solving is both faster and has better asymptotic rate than the dense one.

\begin{figure}[htpb]
\centering
\scalebox{0.8}{
\begin{tikzpicture}

\definecolor{color0}{rgb}{0.12156862745098,0.466666666666667,0.705882352941177}
\definecolor{color1}{rgb}{1,0.498039215686275,0.0549019607843137}
\definecolor{color2}{rgb}{0.172549019607843,0.627450980392157,0.172549019607843}
\definecolor{color3}{rgb}{0.83921568627451,0.152941176470588,0.156862745098039}

\begin{axis}[
legend cell align={left},
legend entries={{H-matrix},{Direct},{$O(N^2)$},{$O(N)$}},
legend style={draw=white!80.0!black},
tick align=outside,
tick pos=left,
title={Construction Time},
x grid style={white!69.01960784313725!black},
xlabel={Matrix Size},
xmin=724.077343935024, xmax=1482910.40037893,
xmode=log,
y grid style={white!69.01960784313725!black},
ylabel={Time (sec)},
ymin=0.00220187706027951, ymax=42434.4875567995,
ymode=log
]
\addlegendimage{no markers, color0}
\addlegendimage{no markers, color1}
\addlegendimage{no markers, color2}
\addlegendimage{no markers, color3}
\addplot [semithick, color0, mark=*, mark size=3, mark options={solid}]
table [row sep=\\]{%
1024	0.016117583 \\
2048	0.046379019 \\
4096	0.127069826 \\
8192	0.292432849 \\
16384	0.860845945 \\
32768	2.081860078 \\
65536	4.223938207 \\
131072	8.916075953 \\
262144	20.242718506 \\
524288	39.780334734 \\
1048576	81.861753704 \\
};
\addplot [semithick, color1, mark=*, mark size=3, mark options={solid}]
table [row sep=\\]{%
1024	0.028535442 \\
2048	0.132978801 \\
4096	1.452989489 \\
8192	3.543617585 \\
16384	35.091152046 \\
32768	159.240163905 \\
65536	642.000945215 \\
};
\addplot [semithick, color2, dashed]
table [row sep=\\]{%
1024	4.83310264519679 \\
2048	19.3324105807872 \\
4096	77.3296423231486 \\
8192	309.318569292594 \\
16384	1237.27427717038 \\
32768	4949.09710868151 \\
65536	19796.3884347261 \\
};
\addplot [semithick, color3, dashed]
table [row sep=\\]{%
1024	0.00471982680194999 \\
2048	0.00943965360389999 \\
4096	0.0188793072078 \\
8192	0.0377586144155999 \\
16384	0.0755172288311999 \\
32768	0.1510344576624 \\
65536	0.302068915324799 \\
131072	0.604137830649599 \\
262144	1.2082756612992 \\
524288	2.4165513225984 \\
1048576	4.83310264519679 \\
};
\end{axis}

\end{tikzpicture}}~
\scalebox{0.8}{
\begin{tikzpicture}

\definecolor{color0}{rgb}{0.12156862745098,0.466666666666667,0.705882352941177}
\definecolor{color1}{rgb}{1,0.498039215686275,0.0549019607843137}
\definecolor{color2}{rgb}{0.172549019607843,0.627450980392157,0.172549019607843}
\definecolor{color3}{rgb}{0.83921568627451,0.152941176470588,0.156862745098039}

\begin{axis}[
legend cell align={left},
legend entries={{H-matrix},{Direct},{$O(N^2)$},{$O(N)$}},
legend style={at={(0.03,0.97)}, anchor=north west, draw=white!80.0!black},
tick align=outside,
tick pos=left,
title={Matrix Vector Multiplication},
x grid style={white!69.01960784313725!black},
xlabel={Matrix Size},
xmin=724.077343935024, xmax=1482910.40037893,
xmode=log,
y grid style={white!69.01960784313725!black},
ylabel={Time (sec)},
ymin=6.99959271769686e-05, ymax=8.21602756609917,
ymode=log
]
\addlegendimage{no markers, color0}
\addlegendimage{no markers, color1}
\addlegendimage{no markers, color2}
\addlegendimage{no markers, color3}
\addplot [semithick, color0, mark=*, mark size=3, mark options={solid}]
table [row sep=\\]{%
1024	0.0004577314 \\
2048	0.0010067572 \\
4096	0.0021811736 \\
8192	0.0071558649 \\
16384	0.0199023826 \\
32768	0.043914754 \\
65536	0.0885521418 \\
131072	0.1830574324 \\
262144	0.378938941 \\
524288	0.8072160345 \\
1048576	1.7325115671 \\
};
\addplot [semithick, color1, mark=*, mark size=3, mark options={solid}]
table [row sep=\\]{%
1024	0.0001189895 \\
2048	0.0004108002 \\
4096	0.002661161 \\
8192	0.0134586426 \\
16384	0.0438625761 \\
32768	0.1996369868 \\
65536	0.8092492205 \\
};
\addplot [semithick, color2, dashed]
table [row sep=\\]{%
1024	0.000483310264519679 \\
2048	0.00193324105807872 \\
4096	0.00773296423231486 \\
8192	0.0309318569292594 \\
16384	0.123727427717038 \\
32768	0.494909710868151 \\
65536	1.97963884347261 \\
};
\addplot [semithick, color3, dashed]
table [row sep=\\]{%
1024	0.00471982680194999 \\
2048	0.00943965360389999 \\
4096	0.0188793072078 \\
8192	0.0377586144155999 \\
16384	0.0755172288311999 \\
32768	0.1510344576624 \\
65536	0.302068915324799 \\
131072	0.604137830649599 \\
262144	1.2082756612992 \\
524288	2.4165513225984 \\
1048576	4.83310264519679 \\
};
\end{axis}

\end{tikzpicture}}
\scalebox{0.8}{
\begin{tikzpicture}

\definecolor{color0}{rgb}{0.12156862745098,0.466666666666667,0.705882352941177}
\definecolor{color1}{rgb}{1,0.498039215686275,0.0549019607843137}
\definecolor{color2}{rgb}{0.172549019607843,0.627450980392157,0.172549019607843}
\definecolor{color3}{rgb}{0.83921568627451,0.152941176470588,0.156862745098039}

\begin{axis}[
legend cell align={left},
legend entries={{H-matrix},{Direct},{$O(N^3)$},{$O(N)$}},
legend style={at={(0.03,0.97)}, anchor=north west, draw=white!80.0!black},
tick align=outside,
tick pos=left,
title={LU},
x grid style={white!69.01960784313725!black},
xlabel={Matrix Size},
xmin=724.077343935024, xmax=1482910.40037893,
xmode=log,
y grid style={white!69.01960784313725!black},
ylabel={Time (sec)},
ymin=0.00224889443808227, ymax=27228.4480947519,
ymode=log
]
\addlegendimage{no markers, color0}
\addlegendimage{no markers, color1}
\addlegendimage{no markers, color2}
\addlegendimage{no markers, color3}
\addplot [semithick, color0, mark=*, mark size=3, mark options={solid}]
table [row sep=\\]{%
1024	0.033150665 \\
2048	0.089554581 \\
4096	0.215134275 \\
8192	0.474807033 \\
16384	1.057615582 \\
32768	2.172390986 \\
65536	4.970435166 \\
131072	10.88718557 \\
262144	23.683602644 \\
524288	59.203750851 \\
1048576	125.118831592 \\
};
\addplot [semithick, color1, mark=*, mark size=3, mark options={solid}]
table [row sep=\\]{%
1024	0.012532439 \\
2048	0.068024574 \\
4096	0.424812388 \\
8192	1.974486374 \\
16384	14.869242377 \\
32768	108.248607665 \\
65536	835.887055881 \\
};
\addplot [semithick, color2, dashed]
table [row sep=\\]{%
1024	0.0494909710868151 \\
2048	0.395927768694522 \\
4096	3.16742214955616 \\
8192	25.3393771964493 \\
16384	202.715017571595 \\
32768	1621.72014057276 \\
65536	12973.761124582 \\
};
\addplot [semithick, color3, dashed]
table [row sep=\\]{%
1024	0.00471982680194999 \\
2048	0.00943965360389999 \\
4096	0.0188793072078 \\
8192	0.0377586144155999 \\
16384	0.0755172288311999 \\
32768	0.1510344576624 \\
65536	0.302068915324799 \\
131072	0.604137830649599 \\
262144	1.2082756612992 \\
524288	2.4165513225984 \\
1048576	4.83310264519679 \\
};
\end{axis}

\end{tikzpicture}}~
\scalebox{0.8}{
\begin{tikzpicture}

\definecolor{color0}{rgb}{0.12156862745098,0.466666666666667,0.705882352941177}
\definecolor{color1}{rgb}{1,0.498039215686275,0.0549019607843137}
\definecolor{color2}{rgb}{0.172549019607843,0.627450980392157,0.172549019607843}
\definecolor{color3}{rgb}{0.83921568627451,0.152941176470588,0.156862745098039}

\begin{axis}[
legend cell align={left},
legend entries={{H-matrix},{Direct},{$O(N^2)$},{$O(N)$}},
legend style={at={(0.03,0.97)}, anchor=north west, draw=white!80.0!black},
tick align=outside,
tick pos=left,
title={Solve},
x grid style={white!69.01960784313725!black},
xlabel={Matrix Size},
xmin=724.077343935024, xmax=1482910.40037893,
xmode=log,
y grid style={white!69.01960784313725!black},
ylabel={Time (sec)},
ymin=0.000252795812344401, ymax=7.7286721108175,
ymode=log
]
\addlegendimage{no markers, color0}
\addlegendimage{no markers, color1}
\addlegendimage{no markers, color2}
\addlegendimage{no markers, color3}
\addplot [semithick, color0, mark=*, mark size=3, mark options={solid}]
table [row sep=\\]{%
1024	0.0006127878 \\
2048	0.0013368844 \\
4096	0.0032873561 \\
8192	0.0069357575 \\
16384	0.0206869708 \\
32768	0.0408734937 \\
65536	0.0881179885 \\
131072	0.177369463 \\
262144	0.3553275244 \\
524288	0.7210637536 \\
1048576	1.491324965 \\
};
\addplot [semithick, color1, mark=*, mark size=3, mark options={solid}]
table [row sep=\\]{%
1024	0.0004042488 \\
2048	0.0014980889 \\
4096	0.0147092541 \\
8192	0.0599093233 \\
16384	0.2397168851 \\
32768	1.0325378595 \\
65536	3.984981354 \\
};
\addplot [semithick, color2, dashed]
table [row sep=\\]{%
1024	0.000483310264519679 \\
2048	0.00193324105807872 \\
4096	0.00773296423231486 \\
8192	0.0309318569292594 \\
16384	0.123727427717038 \\
32768	0.494909710868151 \\
65536	1.97963884347261 \\
};
\addplot [semithick, color3, dashed]
table [row sep=\\]{%
1024	0.00471982680194999 \\
2048	0.00943965360389999 \\
4096	0.0188793072078 \\
8192	0.0377586144155999 \\
16384	0.0755172288311999 \\
32768	0.1510344576624 \\
65536	0.302068915324799 \\
131072	0.604137830649599 \\
262144	1.2082756612992 \\
524288	2.4165513225984 \\
1048576	4.83310264519679 \\
};
\end{axis}

\end{tikzpicture}}
\caption{1D case. \textit{Upper left}: Same as the second plot in \cref{fig:construction}. \textit{Upper right}: Matrix vector multiplication is also much more efficient using the $\mathcal{H}$-matrices than using the dense matrix. It has the asymptotic complexity rate approximately $\mathcal{O}(N)$, compared to $\mathcal{O}(N^2)$ for dense matrices. \textit{Lower left}: LU decomposition of $\mathcal{H}$-LU and the dense LU. The $\mathcal{H}$-LU has linear asymptotic complexity, where the dense LU has complexity $\mathcal{O}(N^3)$.  \textit{Lower right}: Solving time for both the dense matrix and the $\mathcal{H}$-matrix. The $\mathcal{H}$-matrix solving is both faster and has an asymptotic rate that is approximately linear.}
\label{fig:g1d}
\end{figure}
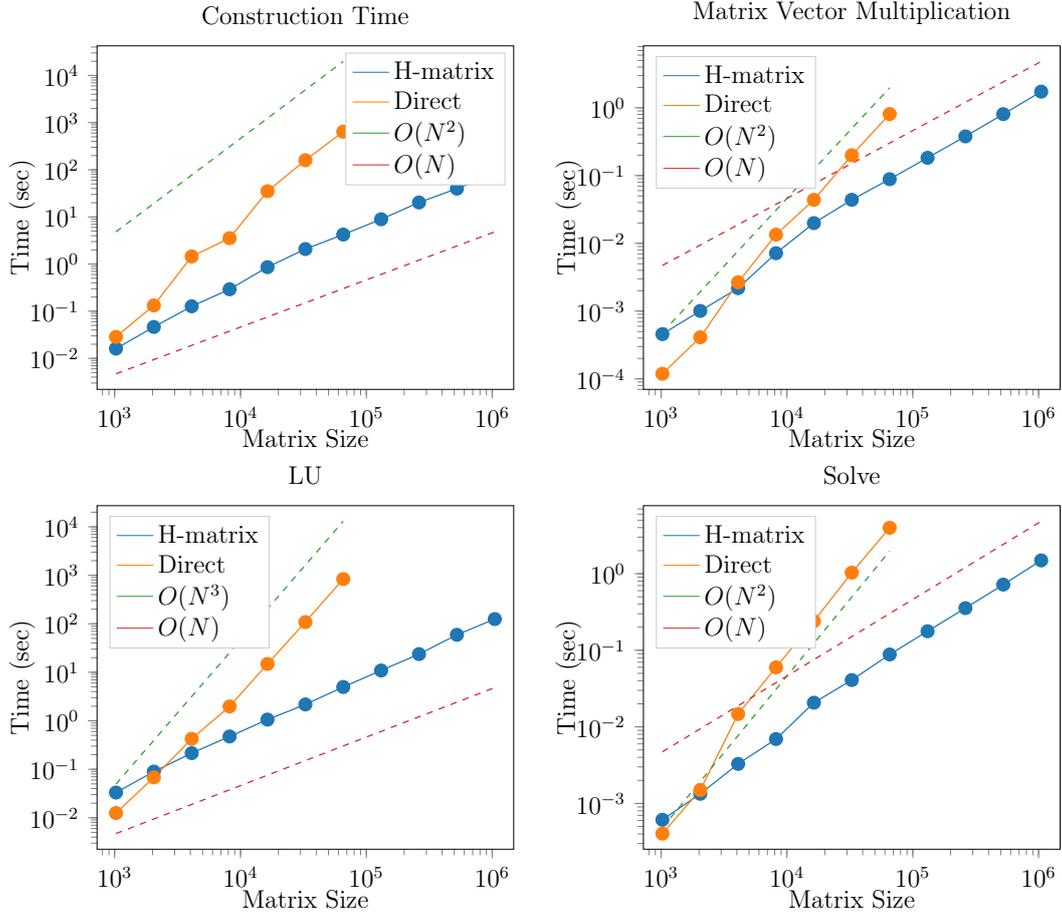

\subsection{Efficiency of $\mathcal{H}$-Matrix: 2D Case}

We mention that the $\mathcal{H}$-matrix technique also works well in 2D. We consider the model problem
\begin{equation}\label{equ:model2d}
	\begin{cases}
		u_t = \int_{\RR^2} (u(\bx+\by)-u(\bx))e^{-5\by^2} d\by & (\bx,t)\in \RR^2\times (0,1]\\
		u(\bx,0) = e^{-50|\bx|^2} &   \bx\in \RR^2
	\end{cases}
\end{equation}

We truncate the computational domain to $[-1,1]$ by imposing the homogeneous Dirichlet boundary condition $u(\bx)=0$, $\forall\bx\in ([-1,1]^2)^c$. In this case, we divide $[-1,1]^2$ into $2^n\times 2^n$ equal size squares and let $h=\frac{1}{2^{n-1}}$. In the construction of the $\mathcal{H}$-matrix, we use a fixed rank strategy and let $r=10$.


We perform the same comparison as that in the last section. \Cref{fig:g2d} shows that the $\mathcal{H}$-matrix technique has better asymptotic rate than that of the dense matrices concerning time consumption. $\mathcal{H}$-matrix will have a great advantage over the dense matrices over the dense matrices for large-scale problems. 
\begin{figure}[htpb]
\centering
\scalebox{0.8}{
\begin{tikzpicture}

\definecolor{color0}{rgb}{0.12156862745098,0.466666666666667,0.705882352941177}
\definecolor{color1}{rgb}{1,0.498039215686275,0.0549019607843137}
\definecolor{color2}{rgb}{0.172549019607843,0.627450980392157,0.172549019607843}
\definecolor{color3}{rgb}{0.83921568627451,0.152941176470588,0.156862745098039}

\begin{axis}[
legend cell align={left},
legend entries={{H-matrix},{Direct},{$O(N^2)$},{$O(N)$}},
legend style={at={(0.03,0.97)}, anchor=north west, draw=white!80.0!black},
tick align=outside,
tick pos=left,
title={Construction Time},
x grid style={white!69.01960784313725!black},
xlabel={Matrix Size},
xmin=168.89701257893, xmax=1589344.0144452,
xmode=log,
y grid style={white!69.01960784313725!black},
ylabel={Time (sec)},
ymin=0.00170187859347138, ymax=150702.501506238,
ymode=log
]
\addlegendimage{no markers, color0}
\addlegendimage{no markers, color1}
\addlegendimage{no markers, color2}
\addlegendimage{no markers, color3}
\addplot [semithick, color0, mark=*, mark size=3, mark options={solid}]
table [row sep=\\]{%
256	0.011541026 \\
1024	0.279334639 \\
4096	1.709158927 \\
16384	9.160934275 \\
65536	48.77634445 \\
262144	272.110659505 \\
1048576	1390.282231436 \\
};
\addplot [semithick, color1, mark=*, mark size=3, mark options={solid}]
table [row sep=\\]{%
1024	0.505909355 \\
4096	9.482454493 \\
16384	82.612928767 \\
65536	1413.516021609 \\
};
\addplot [semithick, color2, dashed]
table [row sep=\\]{%
1024	16.0149105928211 \\
4096	256.238569485137 \\
16384	4099.8171117622 \\
65536	65597.0737881952 \\
};
\addplot [semithick, color3, dashed]
table [row sep=\\]{%
256	0.00390989028145047 \\
1024	0.0156395611258019 \\
4096	0.0625582445032074 \\
16384	0.25023297801283 \\
65536	1.00093191205132 \\
262144	4.00372764820528 \\
1048576	16.0149105928211 \\
};
\end{axis}

\end{tikzpicture}}~
\scalebox{0.8}{
\begin{tikzpicture}

\definecolor{color0}{rgb}{0.12156862745098,0.466666666666667,0.705882352941177}
\definecolor{color1}{rgb}{1,0.498039215686275,0.0549019607843137}
\definecolor{color2}{rgb}{0.172549019607843,0.627450980392157,0.172549019607843}
\definecolor{color3}{rgb}{0.83921568627451,0.152941176470588,0.156862745098039}

\begin{axis}[
legend cell align={left},
legend entries={{H-matrix},{Direct},{$O(N^2)$},{$O(N)$}},
legend style={at={(0.03,0.97)}, anchor=north west, draw=white!80.0!black},
tick align=outside,
tick pos=left,
title={Matrix Vector Multiplication},
x grid style={white!69.01960784313725!black},
xlabel={Matrix Size},
xmin=168.89701257893, xmax=1589344.0144452,
xmode=log,
y grid style={white!69.01960784313725!black},
ylabel={Time (sec)},
ymin=3.63621989286145e-05, ymax=29.7358557499936,
ymode=log
]
\addlegendimage{no markers, color0}
\addlegendimage{no markers, color1}
\addlegendimage{no markers, color2}
\addlegendimage{no markers, color3}
\addplot [semithick, color0, mark=*, mark size=3, mark options={solid}]
table [row sep=\\]{%
256	6.75159e-05 \\
1024	0.0007449968 \\
4096	0.0048214474 \\
16384	0.0404143523 \\
65536	0.2325246504 \\
262144	1.327147896 \\
1048576	6.657581154 \\
};
\addplot [semithick, color1, mark=*, mark size=3, mark options={solid}]
table [row sep=\\]{%
1024	0.0001598307 \\
4096	0.1495839336 \\
16384	0.0454957939 \\
65536	0.8845565423 \\
};
\addplot [semithick, color2, dashed]
table [row sep=\\]{%
1024	0.00160149105928211 \\
4096	0.0256238569485137 \\
16384	0.40998171117622 \\
65536	6.55970737881952 \\
};
\addplot [semithick, color3, dashed]
table [row sep=\\]{%
256	0.00390989028145047 \\
1024	0.0156395611258019 \\
4096	0.0625582445032074 \\
16384	0.25023297801283 \\
65536	1.00093191205132 \\
262144	4.00372764820528 \\
1048576	16.0149105928211 \\
};
\end{axis}

\end{tikzpicture}}
\scalebox{0.8}{
\begin{tikzpicture}

\definecolor{color0}{rgb}{0.12156862745098,0.466666666666667,0.705882352941177}
\definecolor{color1}{rgb}{1,0.498039215686275,0.0549019607843137}
\definecolor{color2}{rgb}{0.172549019607843,0.627450980392157,0.172549019607843}
\definecolor{color3}{rgb}{0.83921568627451,0.152941176470588,0.156862745098039}

\begin{axis}[
legend cell align={left},
legend entries={{H-matrix},{Direct},{$O(N^3)$},{$O(N)$}},
legend style={at={(0.03,0.97)}, anchor=north west, draw=white!80.0!black},
tick align=outside,
tick pos=left,
title={LU},
x grid style={white!69.01960784313725!black},
xlabel={Matrix Size},
xmin=168.89701257893, xmax=1589344.0144452,
xmode=log,
y grid style={white!69.01960784313725!black},
ylabel={Time (sec)},
ymin=0.0020855703201998, ymax=1564.81877132579,
ymode=log
]
\addlegendimage{no markers, color0}
\addlegendimage{no markers, color1}
\addlegendimage{no markers, color2}
\addlegendimage{no markers, color3}
\addplot [semithick, color0, mark=*, mark size=3, mark options={solid}]
table [row sep=\\]{%
256	0.00385728 \\
1024	0.232011202 \\
4096	1.085102725 \\
16384	4.410872722 \\
65536	22.965052433 \\
262144	95.64710674 \\
1048576	452.012639607 \\
};
\addplot [semithick, color1, mark=*, mark size=3, mark options={solid}]
table [row sep=\\]{%
1024	0.045030348 \\
4096	0.676154922 \\
16384	13.478212619 \\
65536	846.072773034 \\
};
\addplot [semithick, color2, dashed]
table [row sep=\\]{%
1024	0.0160149105928211 \\
4096	0.256238569485137 \\
16384	4.0998171117622 \\
65536	65.5970737881952 \\
};
\addplot [semithick, color3, dashed]
table [row sep=\\]{%
256	0.00390989028145047 \\
1024	0.0156395611258019 \\
4096	0.0625582445032074 \\
16384	0.25023297801283 \\
65536	1.00093191205132 \\
262144	4.00372764820528 \\
1048576	16.0149105928211 \\
};
\end{axis}

\end{tikzpicture}}~
\scalebox{0.8}{
\begin{tikzpicture}

\definecolor{color0}{rgb}{0.12156862745098,0.466666666666667,0.705882352941177}
\definecolor{color1}{rgb}{1,0.498039215686275,0.0549019607843137}
\definecolor{color2}{rgb}{0.172549019607843,0.627450980392157,0.172549019607843}
\definecolor{color3}{rgb}{0.83921568627451,0.152941176470588,0.156862745098039}

\begin{axis}[
legend cell align={left},
legend entries={{H-matrix},{Direct},{$O(N^2)$},{$O(N)$}},
legend style={at={(0.03,0.97)}, anchor=north west, draw=white!80.0!black},
tick align=outside,
tick pos=left,
title={Solve},
x grid style={white!69.01960784313725!black},
xlabel={Matrix Size},
xmin=168.89701257893, xmax=1589344.0144452,
xmode=log,
y grid style={white!69.01960784313725!black},
ylabel={Time (sec)},
ymin=0.000229767637647198, ymax=119.309419715404,
ymode=log
]
\addlegendimage{no markers, color0}
\addlegendimage{no markers, color1}
\addlegendimage{no markers, color2}
\addlegendimage{no markers, color3}
\addplot [semithick, color0, mark=*, mark size=3, mark options={solid}]
table [row sep=\\]{%
256	0.0016703639 \\
1024	0.00119636 \\
4096	0.0067007124 \\
16384	0.0265675788 \\
65536	0.1035973081 \\
262144	0.3101372585 \\
1048576	1.2069433433 \\
};
\addplot [semithick, color1, mark=*, mark size=3, mark options={solid}]
table [row sep=\\]{%
1024	0.0004179065 \\
4096	0.0173444921 \\
16384	0.2419952142 \\
65536	4.0719940769 \\
};
\addplot [semithick, color2, dashed]
table [row sep=\\]{%
1024	0.0160149105928211 \\
4096	0.256238569485137 \\
16384	4.0998171117622 \\
65536	65.5970737881952 \\
};
\addplot [semithick, color3, dashed]
table [row sep=\\]{%
256	0.00390989028145047 \\
1024	0.0156395611258019 \\
4096	0.0625582445032074 \\
16384	0.25023297801283 \\
65536	1.00093191205132 \\
262144	4.00372764820528 \\
1048576	16.0149105928211 \\
};
\end{axis}

\end{tikzpicture}}
\caption{2D case. Comparison of construction time, matrix-vector multiplication time, LU decomposition, and solving for both dense matrices as well as $\mathcal{H}$-matrices. The $\mathcal{H}$-matrix technique has linear asymptotic rate. The green dashed line shows the theoretical complexity asymptotic rate for dense matrices while the red line represents the theoretical complexity asymptotic rate for $\mathcal{H}$-matrix.}
\label{fig:g2d}
\end{figure}
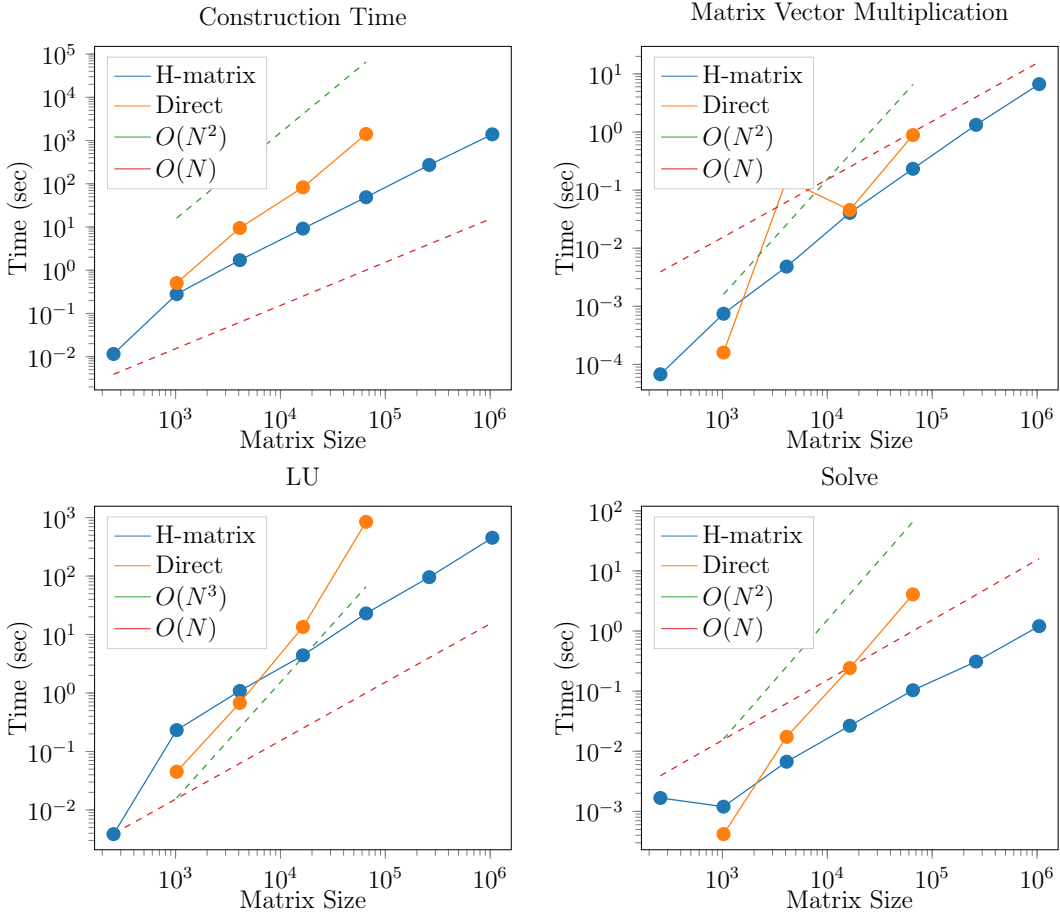

We finally investigate the accuracy of the overall numerical scheme. \Cref{fig:g2d_dxdt} shows the convergence plots for the 2D model \cref{equ:model2d}. Similar to the 1D case, we see second order convergence in time and first order convergence in space.

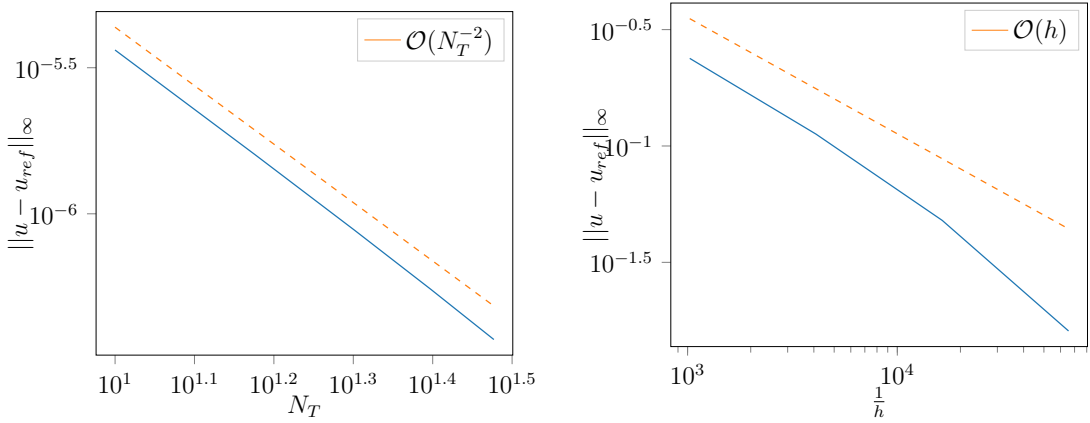
\begin{figure}[htpb]
\centering
\scalebox{0.8}{
\begin{tikzpicture}

\definecolor{color0}{rgb}{0.12156862745098,0.466666666666667,0.705882352941177}
\definecolor{color1}{rgb}{1,0.498039215686275,0.0549019607843137}

\begin{axis}[
legend cell align={left},
legend entries={{$\mathcal{O}(N_T^{-2})$}},
legend style={draw=white!80.0!black},
tick align=outside,
tick pos=left,
x grid style={white!69.01960784313725!black},
xlabel={$N_T$},
xmin=9.46550822640159, xmax=31.6940192564861,
xmode=log,
y grid style={white!69.01960784313725!black},
ylabel={$||u-u_{ref}||_\infty$},
ymin=3.28324272508207e-07, ymax=4.92216172063949e-06,
ymode=log
]
\addlegendimage{no markers, color1}
\addplot [semithick, color0, mark=asterisk*, mark size=3, mark options={solid}, forget plot]
table [row sep=\\]{%
10	3.63641641440426e-06 \\
15	1.59556247035508e-06 \\
20	8.81397490049007e-07 \\
25	5.50864627330006e-07 \\
30	3.71322143824515e-07 \\
};
\addplot [semithick, color1, dashed]
table [row sep=\\]{%
10	4.35219173694216e-06 \\
15	1.93430743864096e-06 \\
20	1.08804793423554e-06 \\
25	6.96350677910747e-07 \\
30	4.83576859660241e-07 \\
};
\end{axis}

\end{tikzpicture}}~
\scalebox{0.8}{
\begin{tikzpicture}

\definecolor{color0}{rgb}{0.12156862745098,0.466666666666667,0.705882352941177}
\definecolor{color1}{rgb}{1,0.498039215686275,0.0549019607843137}

\begin{axis}[
legend cell align={left},
legend entries={{$\mathcal{O}(h)$}},
legend style={draw=white!80.0!black},
tick align=outside,
tick pos=left,
x grid style={white!69.01960784313725!black},
xlabel={$\frac{1}{h}$},
xmin=831.746453868785, xmax=80684.2802729723,
xmode=log,
y grid style={white!69.01960784313725!black},
ylabel={$||u-u_{ref}||_\infty$},
ymin=0.0137424356450106, ymax=0.411334444864694,
ymode=log
]
\addlegendimage{no markers, color1}
\addplot [semithick, color0, mark=asterisk*, mark size=3, mark options={solid}, forget plot]
table [row sep=\\]{%
1024	0.23793984337199 \\
4096	0.11245101302681 \\
16384	0.0478750325889632 \\
65536	0.0160383834381034 \\
};
\addplot [semithick, color1, dashed]
table [row sep=\\]{%
1024	0.352450554567716 \\
4096	0.176225277283858 \\
16384	0.088112638641929 \\
65536	0.0440563193209645 \\
};
\end{axis}

\end{tikzpicture}}
\caption{2D case \cref{equ:model2d}. Similar to the 1D case, we see second order convergence in time and first order convergence in space.}
\label{fig:g2d_dxdt}
\end{figure}

\subsection{Singular and Slow Decaying L\'evy Measure}

Finally, we consider the case where $\nu(y)$ grows to infinity at $y=0$ and has a heavy tail. The case is quite challenging and extensively studied by the community nowadays. For simplicity, we will consider the specific case where $\nu(\bx) = \frac{C_{d, s}}{|\bx|^{d+2s}}$, i.e., the fractional Laplacian~\cite{lischke2018fractional}. Fortunately, we can compute the analytical nonlocal derivative or gradient for some functions. 

For the first example, we consider $u(x)=\exp(-x^2)$ in 1D. Then we have~\cite{huang2016finite}
\begin{equation}\label{equ:u0}
    (-\Delta)^s u(0) = 2^{2s}\Gamma\left( \frac{1+2s}{2} \right)/\sqrt{\pi}
\end{equation}

For this example, since $u(x)$ decay to zero exponentially, we can assume that the far-field interaction $f^{L_W}_x\equiv 0$ for sufficiently large $L_W$. The numerical value is computed using \cref{equ:I1approx,equ:I2approx,equ:I3approx} and compared with the exact value \cref{equ:u0}. The parameters are: $L_W=5.0$, $r=0.2$. The convergence plot is shown in \cref{fig:fig5}. We can see that the error converges like or better than $\mathcal{O}(h^2)$. 

\begin{figure}[H] 
\centering
\includegraphics[width=0.7\textwidth,keepaspectratio]{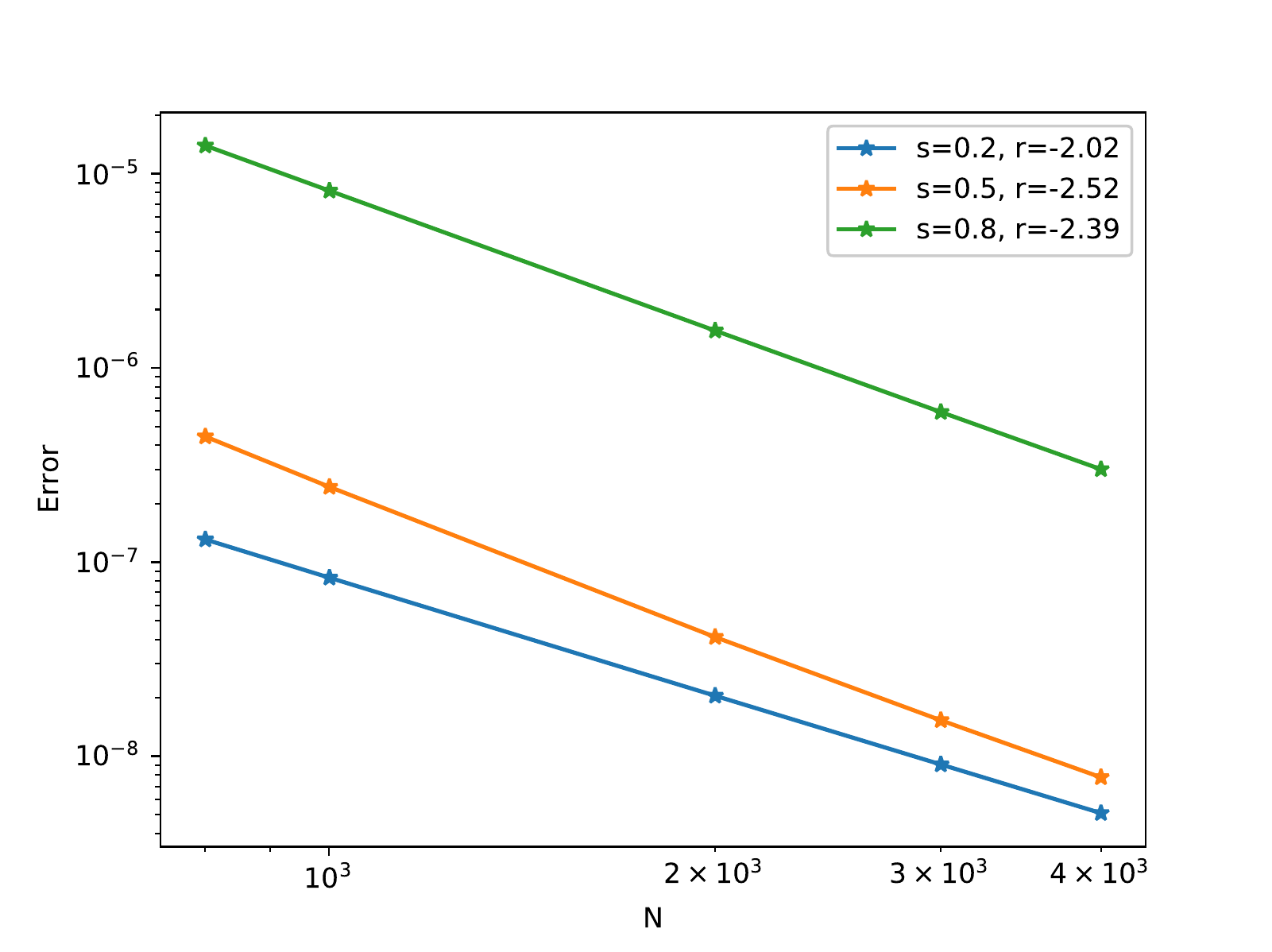}
\caption{Numerical error for approximating \cref{equ:u0}. The error converges like or better than $\mathcal{O}(h^2)$. Here $r$ is the convergence rate.}
\label{fig:fig5}
\end{figure}

We also test the scheme on a challenging problem: the fractional Poisson problem. The PDE
\begin{equation}
    \begin{cases}
        (-\Delta)^s u(x) = 1 & x\in [-1,1]\\
        u(x) = 0 & x\not\in [-1,1]
    \end{cases}
\end{equation}
has a unique solution
\begin{equation}
    u(x) = \frac{2^{2s}\Gamma(1+s)\Gamma\left( \frac{1+2s}{2} \right) }{\Gamma(1/2)} (1-x^2)^s
\end{equation}
Note that $u(x)$ is not smooth across the boundary. In fact, it only belongs to $C^{0,s}([-1,1])$, the $s$-order H\"older space. Numerical algorithms usually exhibit reduced convergence.  We use $L=1.0$ and $L_W=2.0$ so that the support of $u(x)$ is included in the near-field or local interaction. Thus we have $f_x^{L_W}=0$. Since the current implementation only supports forward computation of the nonlocal operator, i.e., given function values, the nonlocal derivative or gradient is computed, we resort to a conjugate gradient approach for recovering $u(x)$ in $[-L,L]$. 

\Cref{fig:fig6} presents the finite difference result obtained from our discretization. We can see that the convergence order is $1.0$ or less, much worse than the Poisson problem where $\mathcal{O}(h^2)$ convergence rate is typical. We need to emphasize this is a universal problem faced by many fractional Laplacian models if a simple truncation method is used. 

\begin{figure}[H] 
\centering
\includegraphics[width=0.45\textwidth,keepaspectratio]{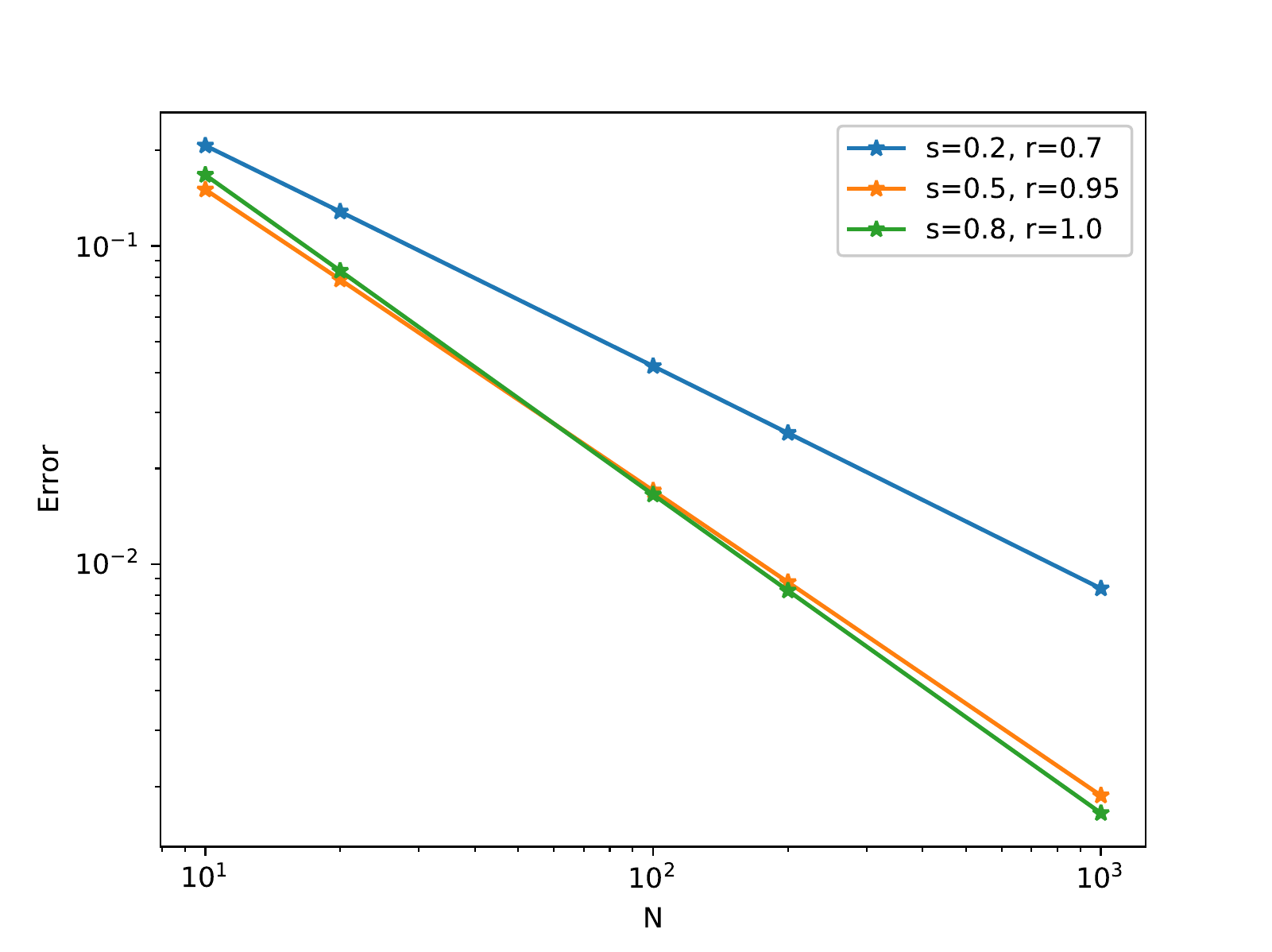}
\includegraphics[width=0.45\textwidth,keepaspectratio]{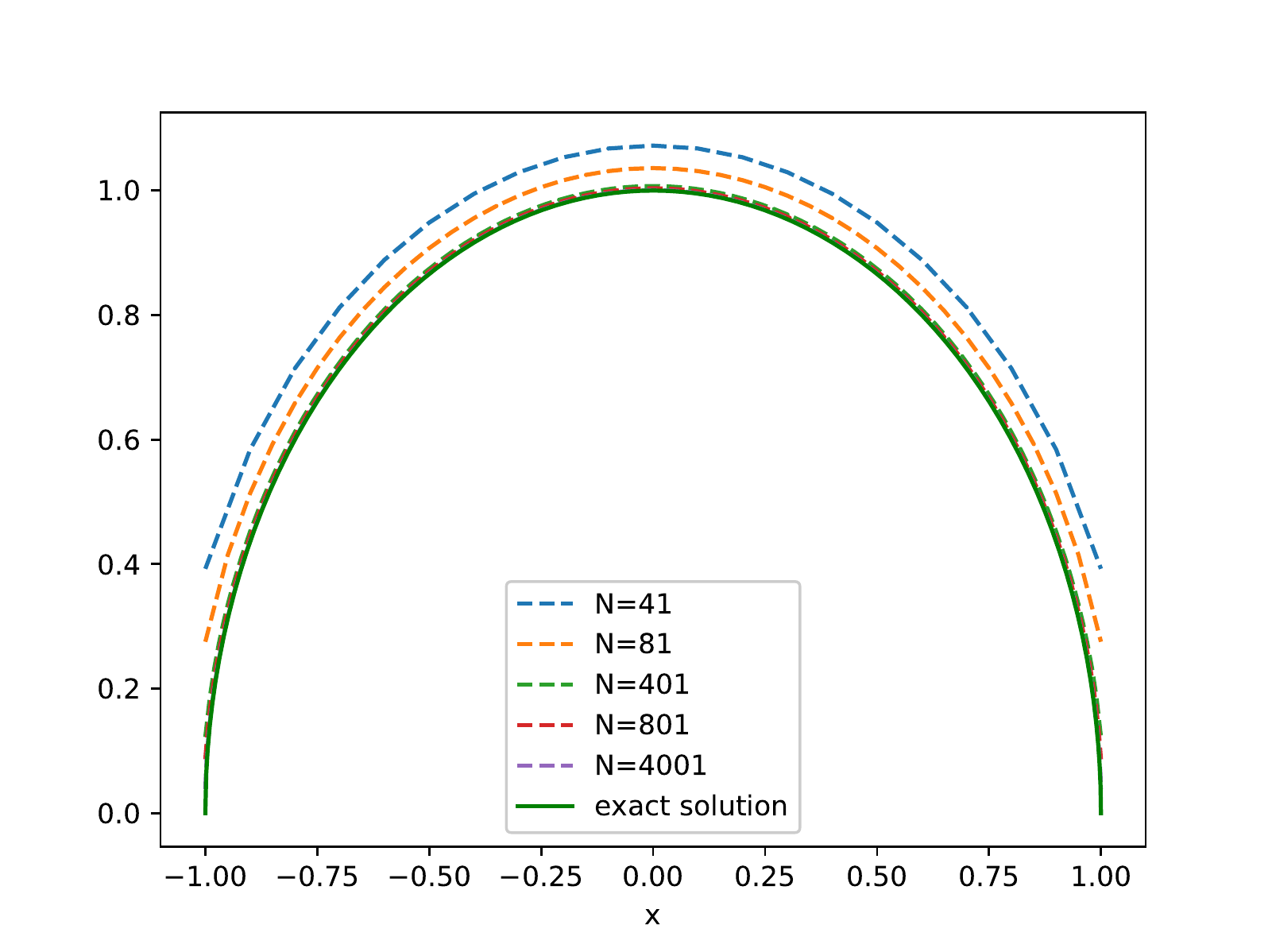}
\caption{The finite difference result obtained from our discretization. The convergence order is $1.0$ or less, much worse than the Poisson problem where $\mathcal{O}(h^2)$ convergence rate is typical}
\label{fig:fig6}
\end{figure}

Finally, we also consider the computation of $(-\Delta)^s u(\bx)$ in 2D, where~\cite{huang2016finite}
\begin{equation}
    u(\bx) = \frac{1}{2^{2s}\Gamma(1+s)^2} (1-|\bx|^2)^s_+
\end{equation}

The analytically result is known for $|\bx|\leq 1$, which is
\begin{equation}\label{equ:u1}
    (-\Delta)^s u(\bx) = 1\quad |\bx|\leq 1
\end{equation}
The numerical result is shown in \cref{fig:fig8}. Near the boundary, due to the non-smoothness of $u(\bx)$, the algorithm has a hard time computing the nonlocal gradient, and therefore we see the oscillatory behavior. However, the computation for the region near the center is good, which does not suffer much from the far-away contribution from nonsmooth boundaries. In the center, the error is only $4\times 10^{-2}$. We used $L_W=2.0$ and $L=1.0$ in this case.  

\begin{figure}[H] 
\centering
\includegraphics[width=0.8\textwidth,keepaspectratio]{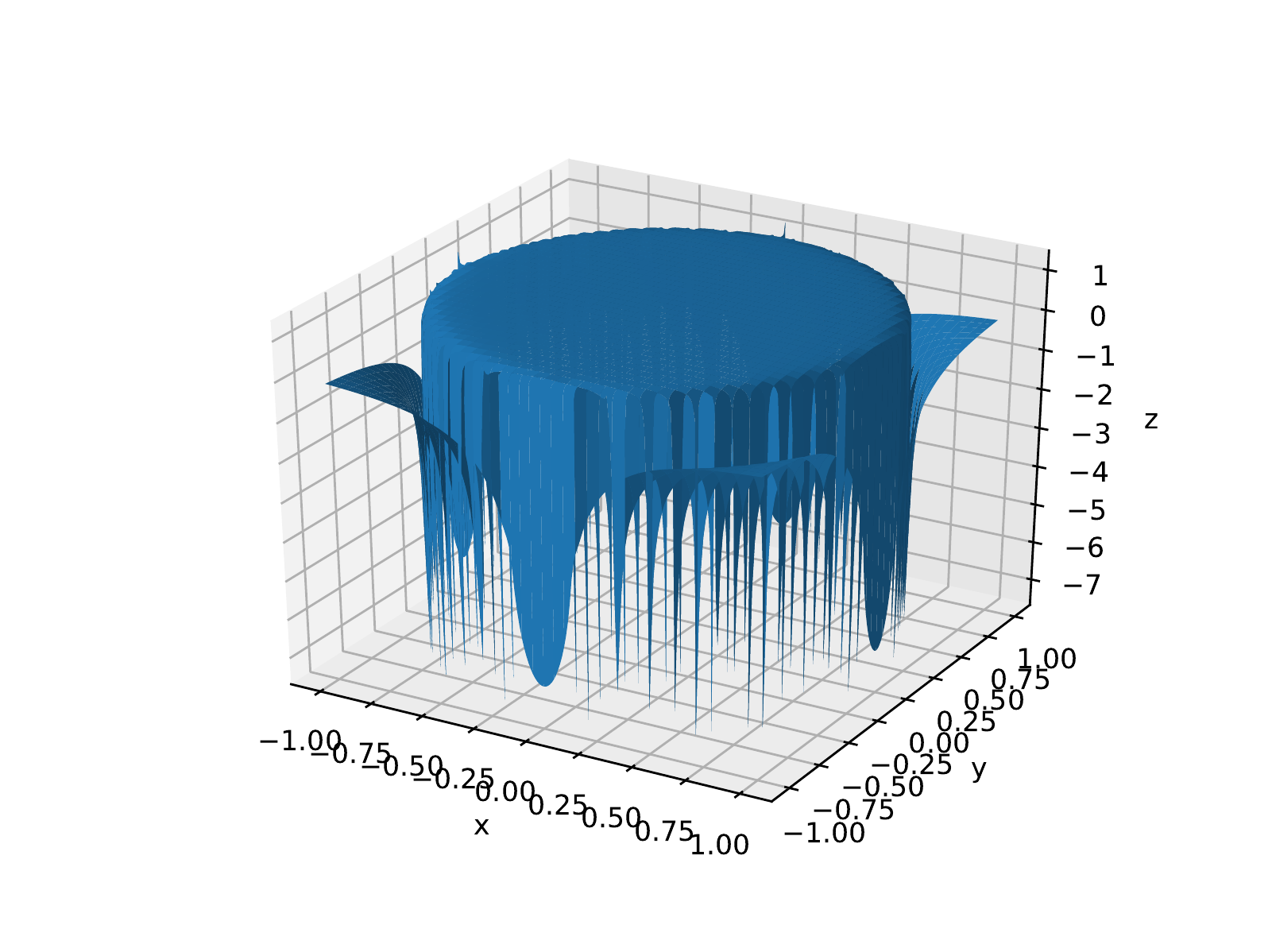}
\caption{Numerical evaluation of \cref{equ:u1}. Near the boundary, due to the nonsmoothness of $u(\bx)$, the algorithm is having a hard time computing the nonlocal gradient and therefore we see the oscillatory behavior. However, the computation for the region near the center is  good, which does not suffer much from the far-away contribution from nonsmooth boundaries. In the center the error is only $4\%$.}
\label{fig:fig8}
\end{figure}

These numerical examples demonstrate that the numerical scheme also works for $\nu(y)$ which has heavy tails. 

\subsection{Application: Variable-Order Fractional Poisson Equation}

In this section, we consider a variable-order space-fractional Poisson equation on a L-shaped domain. 
\begin{equation}\label{equ:2dmodel}
\begin{cases}
	-(-\Delta)^{s(\bx)}u(\bx) = f(\bx)& \bx\in \Omega\\
	u(\bx) = 0& \bx\in \Omega^c
\end{cases}
\end{equation}
Here $\Omega=[-1,1]^2\backslash [0,1]^2$ and
\begin{align}
	s(\bx) =& 0.9-0.8d(\bx)\\
	f(\bx) = & e^{-10\|\bx-\bx_0\|^2} + e^{-10\|\bx-\bx_1\|^2} + e^{-10\|\bx-\bx_2\|^2} 
\end{align}
where $d(x)$ is the distance between $\bx$ and $\partial\Omega$ and 
\begin{equation}
	\bx_0 = \begin{bmatrix}
		-0.5\\
		0.5
	\end{bmatrix}\quad\bx_1 = \begin{bmatrix}
		0.5\\
		-0.5
	\end{bmatrix}\quad
\bx_2 = \begin{bmatrix}
		-0.5\\
		-0.5
	\end{bmatrix}
\end{equation}
Note that $s(x)\in (0,1)$. \Cref{fig:fs} shows the plot of $f(x)$ and $s(x)$. 

\begin{figure}[htpb]
\centering
\includegraphics[width=0.5\textwidth,keepaspectratio]{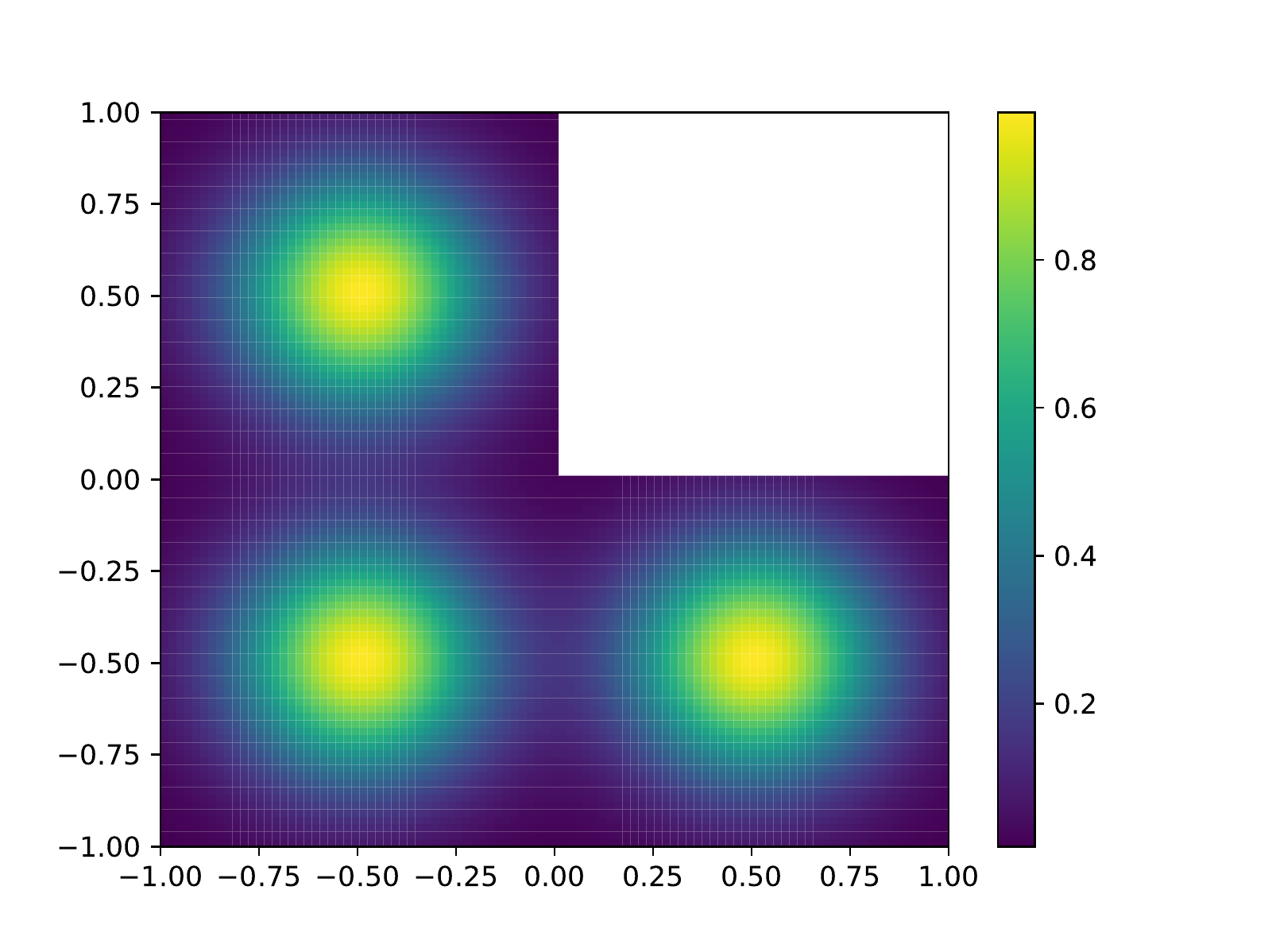}~
\includegraphics[width=0.5\textwidth,keepaspectratio]{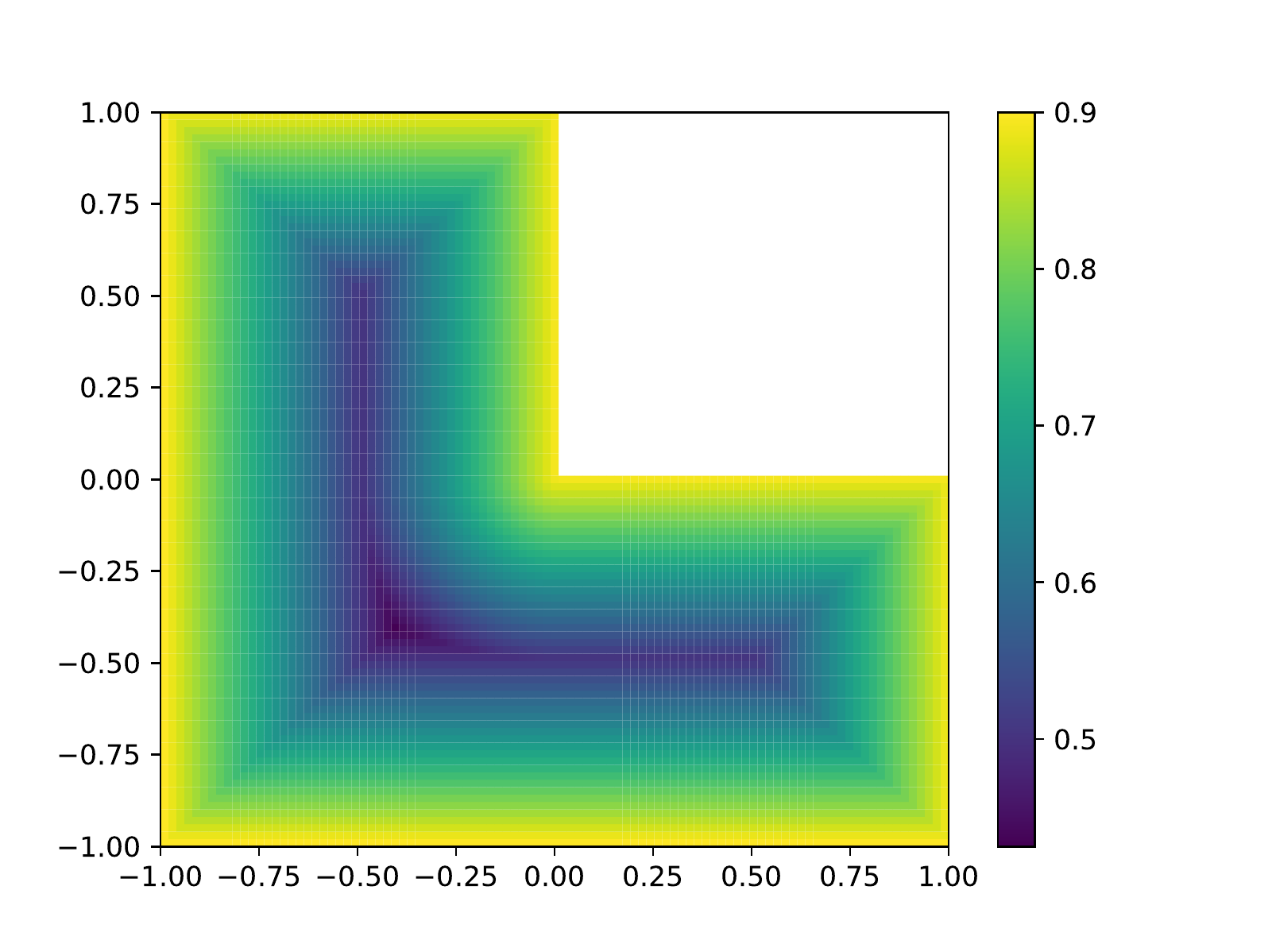}
\caption{Left: The source function $f(x)$ used in the model \cref{equ:2dmodel}; right: the variable fractional index $s(x)$ used in the same model.}
\label{fig:fs}
\end{figure}

We apply the numerical discretization proposed in \cref{sect:extension} with uniform grids and obtained a linear system as follows
\begin{equation}\label{equ:Axb}
	\bA \bu = \bbf
\end{equation}
Due to the non-locality of the fractional Laplacian operator, the stiffness matrix $\bA$ is a dense matrix. The dense LU method becomes infeasible as the problem size increases. An iterative solver becomes desirable in this situation. However, as the problem size becomes larger, the condition number becomes worse and iterative solvers without proper preconditioning converge very slowly for large scale problems. We proposed the $\mathcal{H}$-LU preconditioner and demonstrated its effectiveness for this problem. Our algorithm is able to find the preconditioner ``automatically'' given only the dense matrix $\bA$ and discretization point $\bx$. The users only have two parameters to tune: $\varepsilon_1$, which essentially determines the truncation threshold for low rank matrix representation, accuracy for $\mathcal{H}$-matrix construction; and $\varepsilon_2$, which determines the compression accuracy for low-rank matrix addition, accuracy for $\mathcal{H}$-LU. We use $\varepsilon_1=10^{-4}$ and $\varepsilon_2=10^{-10}$ for the following numerical experiments.

The algorithm will first reorder the system and divide the discretization points into groups so that the rows/columns corresponding points in the same group will be adjacent in the reordered algebraic system. The reordering is done recursively by K-means with two clusters. \Cref{fig:D40} shows groups of points after reordering by the K-means algorithm. Each color represents an individual group.

\begin{figure}[htpb]
\centering
\scalebox{1.0}{\input{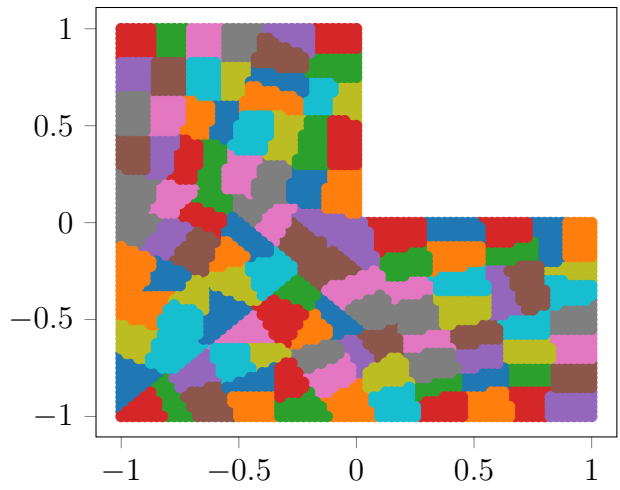}}
\caption{Groups of points after reordering by the K-means algorithm. Each color represents an individual group.}
\label{fig:D40}
\end{figure}

The choice of $\varepsilon_1$ is very important since it controls the tradeoff between accuracy and construction~(and LU) cost for the $\mathcal{H}$-matrix representation. \Cref{fig:M40} shows the constructed $\mathcal{H}$-matrix $H$ for a $4961\times 4961$ matrix and $\varepsilon_1=10^{-4}$. 

\begin{figure}[htpb]
\centering
\includegraphics[width=0.8\textwidth,keepaspectratio]{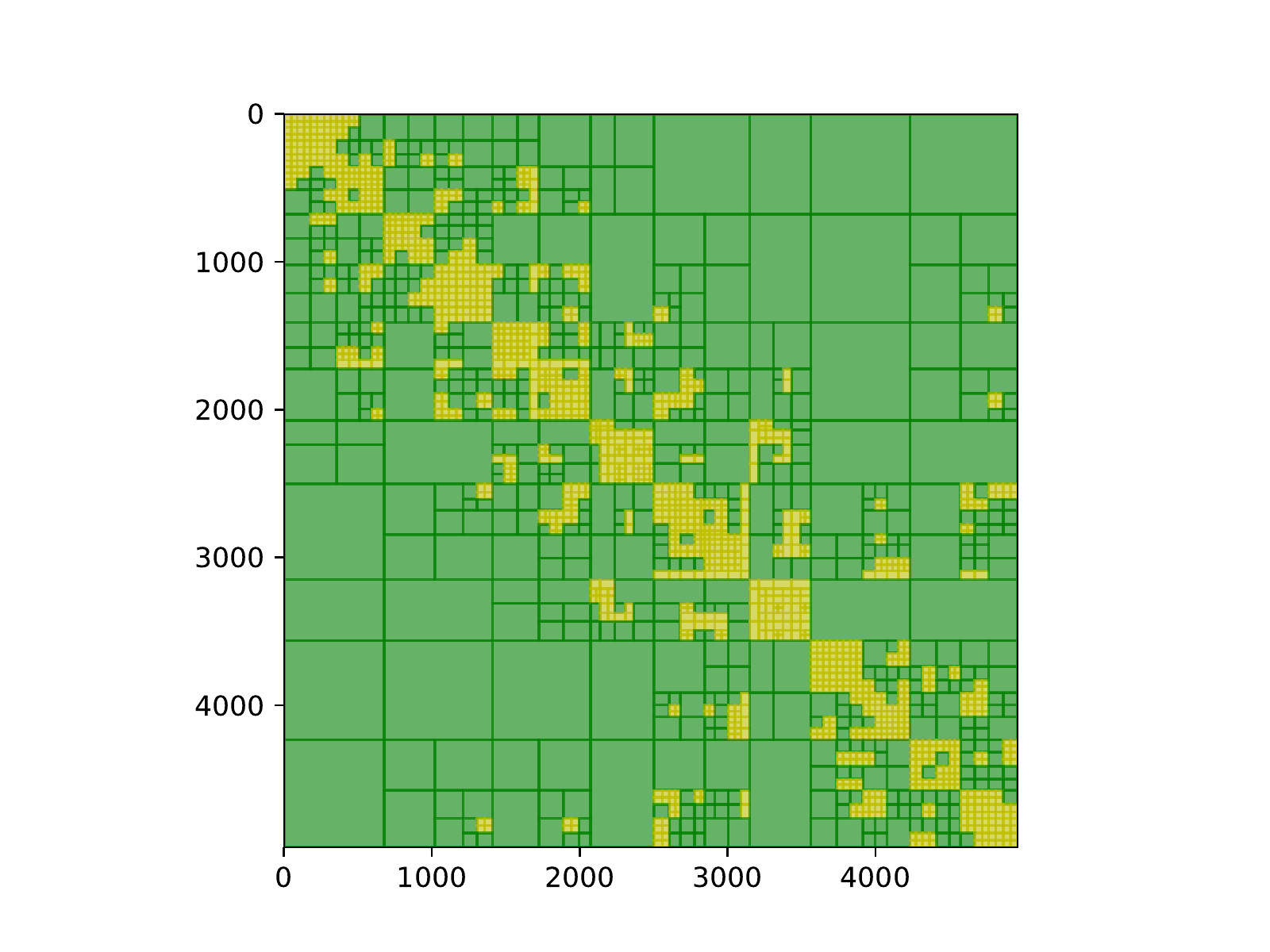}
\caption{$\mathcal{H}$-matrix for a $4961\times 4961$ matrix and $\varepsilon_1=10^{-4}$.}
\label{fig:M40}
\end{figure}

We consider solving \cref{equ:Axb} with preconditioner $H^{-1}$~(after LU factorization) and without. In \Cref{fig:iter40}, the left plot shows the convergence for these two scenarios. In both cases, we measure the error of the solution $\bx^k$ at $k$-th iteration by the relative error formula
\begin{equation}
	e_k = \frac{\|\bA \bx^k - \bbf\|_2}{\|\bbf\|_2}
\end{equation}
 On the right, we also compare the wall time for $\bA^{-1}\bbf$ and $H^{-1}\bbf$. For fairness, $\bA$ is first factorized. The comparison shows that the $\mathcal{H}$-LU preconditioner is also much more efficient than LU preconditioner, especially for large-scale problems and the cases where we need to solve for many different $\bbf$'s.

 \begin{figure}[htpb]
\centering
\scalebox{0.8}{\input{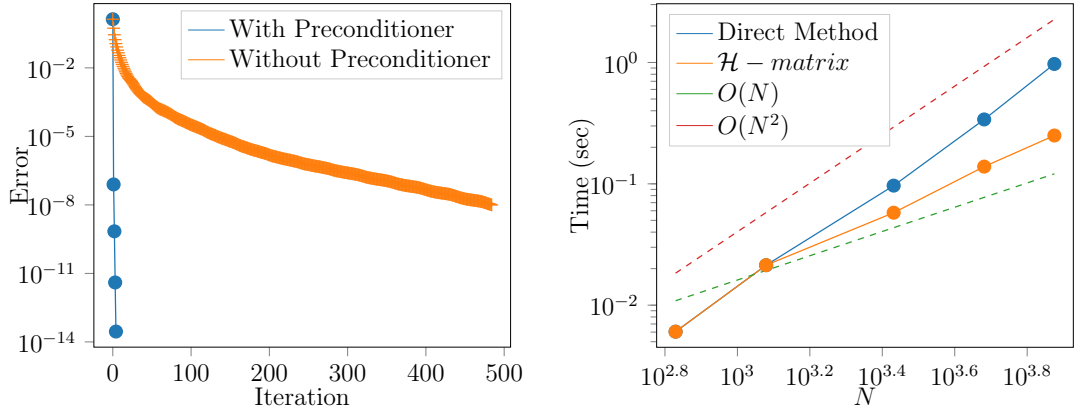}}~
\scalebox{0.8}{
\begin{tikzpicture}

\definecolor{color0}{rgb}{0.12156862745098,0.466666666666667,0.705882352941177}
\definecolor{color1}{rgb}{1,0.498039215686275,0.0549019607843137}
\definecolor{color2}{rgb}{0.172549019607843,0.627450980392157,0.172549019607843}
\definecolor{color3}{rgb}{0.83921568627451,0.152941176470588,0.156862745098039}

\begin{axis}[
legend cell align={left},
legend entries={{Direct Method},{$\mathcal{H}-matrix$},{$O(N)$},{$O(N^2)$}},
legend style={at={(0.03,0.97)}, anchor=north west, draw=white!80.0!black},
tick align=outside,
tick pos=left,
xlabel={$N$},
ylabel={Time (sec)},
x grid style={white!69.01960784313725!black},
xmin=598.433501631519, xmax=8459.58654754124,
xmode=log,
y grid style={white!69.01960784313725!black},
ymin=0.00449806245914075, ymax=3.04829194209777,
ymode=log
]
\addlegendimage{no markers, color0}
\addlegendimage{no markers, color1}
\addlegendimage{no markers, color2}
\addlegendimage{no markers, color3}
\addplot [semithick, color0, mark=*, mark size=3, mark options={solid}]
table [row sep=\\]{%
675	0.00604935 \\
1200	0.0213585 \\
2700	0.0966806 \\
4800	0.339401 \\
7500	0.972231 \\
};
\addplot [semithick, color1, mark=*, mark size=3, mark options={solid}]
table [row sep=\\]{%
675	0.00604935 \\
1200	0.0213585 \\
2700	0.0578209 \\
4800	0.138662 \\
7500	0.250155 \\
};
\addplot [semithick, color2, dashed]
table [row sep=\\]{%
675	0.0108796409922398 \\
1200	0.0193415839862041 \\
2700	0.0435185639689592 \\
4800	0.0773663359448165 \\
7500	0.120884899913776 \\
};
\addplot [semithick, color3, dashed]
table [row sep=\\]{%
675	0.0183593941744047 \\
1200	0.0580247519586123 \\
2700	0.293750306790475 \\
4800	0.928396031337798 \\
7500	2.2665918733833 \\
};
\end{axis}

\end{tikzpicture}}
\caption{Left: the convergence for these two scenarios; right: the wall time for $\bA^{-1}\bbf$ and $H^{-1}\bbf$. For fairness, $\bA$ is first factorized.}
\label{fig:iter40}
\end{figure}

Finally, we show the solution error of approximating $\bA$ by $H$. The error is computed using 
\begin{equation}\label{equ:relative_error}
	e = \frac{\|H^{-1}\bbf - \bA^{-1} \bbf\|_2}{\|\bA^{-1} \bbf\|_2}
\end{equation}
\Cref{fig:sol100} shows the solution $H^{-1}\bbf$ with $100$ points per dimension and relative error \cref{equ:relative_error} agains different problem sizes. We can see that the relative error remains stable and does not increase much as problem size increases, which demonstrates the validity of the $\mathcal{H}$-matrix approximation. 

 \begin{figure}[htpb]
\centering
\includegraphics[width=0.48\textwidth]{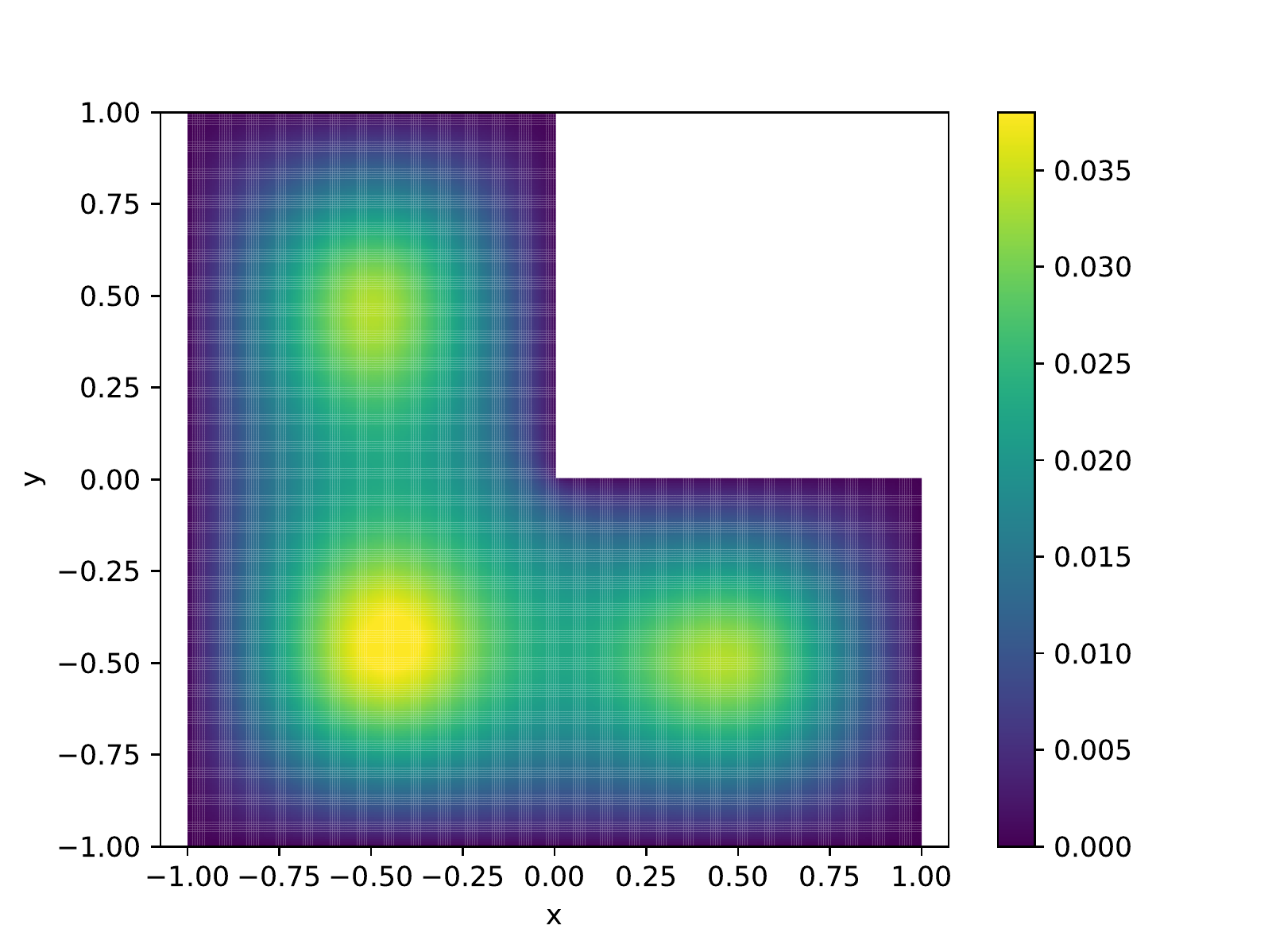}
\scalebox{0.8}{
\begin{tikzpicture}

\definecolor{color0}{rgb}{0.12156862745098,0.466666666666667,0.705882352941177}

\begin{axis}[
tick align=outside,
tick pos=left,
x grid style={white!69.01960784313725!black},
xlabel={$N$},
xmin=333.75, xmax=7841.25,
y grid style={white!69.01960784313725!black},
ylabel={Relative Error},
ymin=0, ymax=0.005
]
\addplot [semithick, color0, mark=*, mark size=3, mark options={solid}, forget plot]
table [row sep=\\]{%
675	0.00135189 \\
1200	0.00167465 \\
2700	0.00180417 \\
4800	0.00133042 \\
7500	0.0016748 \\
};
\end{axis}

\end{tikzpicture}}
\caption{Left: the solution $H^{-1}\bbf$ with $100$ points per dimension; right: relative error \cref{equ:relative_error} for different problem sizes}
\label{fig:sol100}
\end{figure}
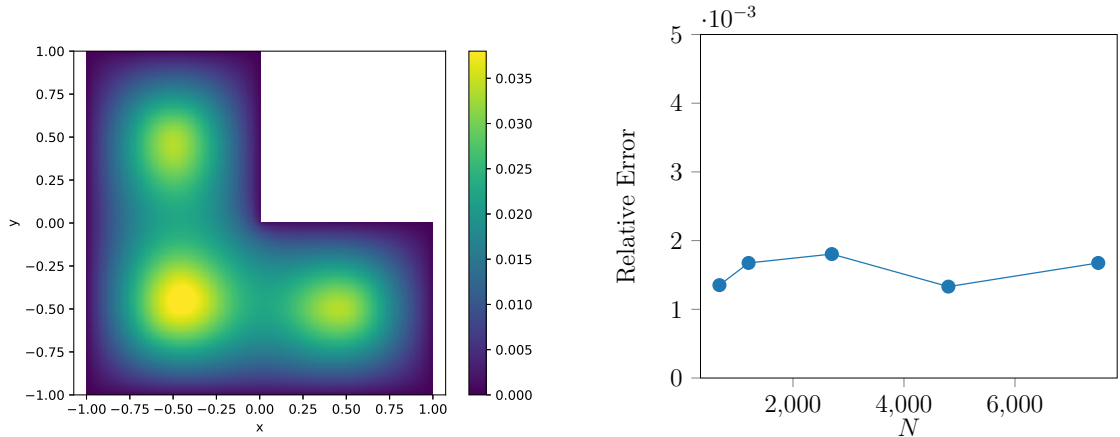

\paragraph{Parallel Assembling}

Since the stiffness matrix can be computed independently and therefore embarrassingly parallelizable. We take advantage of the built-in distributed computing features of \texttt{julia} and assemble the stiffness matrix in parallel\footnote{We used the functions \texttt{remotecall} and \texttt{fetch} for master-worker communication.}. First, the mesh is split into 30 patches~(using K-means or randomly); then each worker is in charge of computing the coefficients for the corresponding rows~(there is a one-to-one correspondence between points on the grids and rows in the matrix). The results are sent to the master machine and assembled into a large dense coefficient matrix. \Cref{fig:parallel} shows the parallel pipeline for assembling the stiffness matrix in \cref{equ:2dmodel}.  A good balance should be struck between data exchange and computation workload. There are opportunities for construction of the $\mathcal{H}$-matrix on the fly and in parallel given the patches; it will be left for future research.

 \begin{figure}[htpb]
\centering
\includegraphics[width=0.8\textwidth]{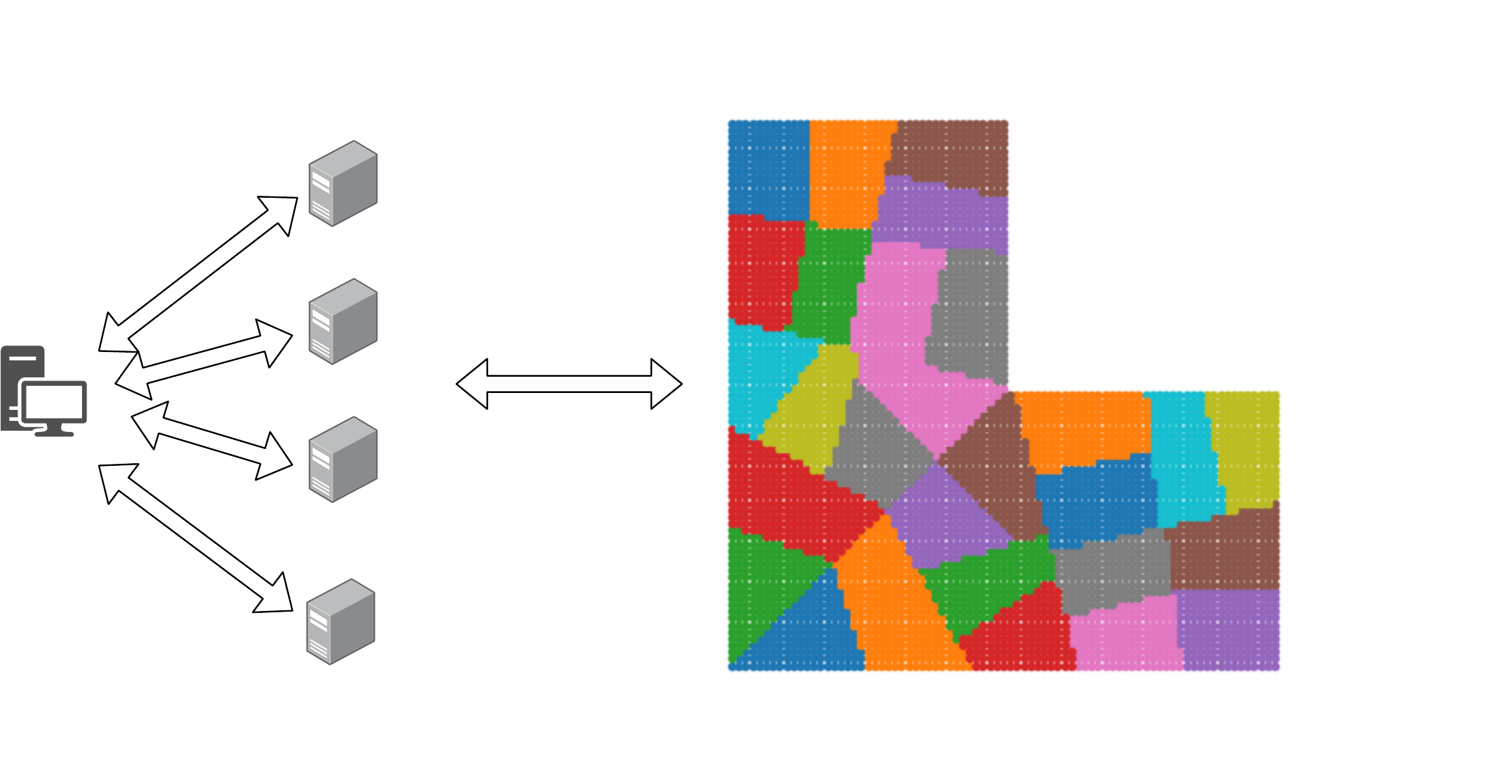}
\caption{The parallel pipeline for assembling the stiffness matrix in \cref{equ:2dmodel}.}
\label{fig:parallel}
\end{figure}

\section{Conclusion}

In this paper, we presented the $\mathcal{H}$-matrix solver for the convection diffusion equation driven by the L\'evy process. We consider both semi-heavy L\'evy measure $\nu(y)<\infty$ as well as the challenging case $\nu(y)\rightarrow 0$, $y\rightarrow 0$, and $\nu(y)$ decays only algebraically. Particularly, when $\nu(y) = \frac{c_{1,s}}{|y|^{1+2s}}$, we recover the so-called fractional Laplacian operator $(-\Delta)^s u(x) =\mathrm{p.v.} \int_{\RR} (u(x+y)-u(x))\nu(y)dy$. In the case $\nu(y)$ is smooth for large $y$, the corresponding coefficient matrices in the explicit or implicit scheme can be efficiently represented by $\mathcal{H}$-matrix. We implemented $\mathcal{H}$-LU and use it as a preconditioner or a direct solver for the convection diffusion equation. Numerical methods demonstrate that the $\mathcal{H}$-matrix is highly efficient compared to the dense matrices for these tasks. 

The algorithms proposed in this paper can also be easily generalized to higher dimensions. To demonstrate, we also present the two-dimensional cases in this paper, which also shows an advantage over direct methods, especially for large-scale problems. 

The convection-diffusion equation or other counterparts driven by the L\'evy process is challenging due to the non-locality of the jump diffusion. This will lead to dense coefficients matrices which makes computation prohibitive for large-scale problems. However, the main finding in this paper shows that by adopting the well-established $\mathcal{H}$-matrix technique, large-scale simulation becomes possible and efficient. Particularly, we have applied the proposed algorithm to solve a variable index fractional Poisson equation, which shows the accuracy and efficiency of the algorithm. 

The code for the paper is available from the authors upon request.

\begin{appendices}
\lib{L\'evy Process \label{sect:levy}}
In this section, we review the basics of the L\'evy process. For general treatment on this topic, refer to \cite{doney2007introduction,papapantoleon2008introduction,sato1999levy}

Consider a given probability space $(\Omega, \mathcal{F}, P)$. A L\'evy process $\{X_t\}_{t\geq 0}$ taking values in $\RR^d$ is defined as a stochastic process with stationary and independent increments. In addition, we assume $X_0=0$ with probability 1. 

By independent, we mean for any distinct time $0\leq t_1<t_2<\ldots<t_n$, we have $X_{t_1}, X_{t_2-t_1}, \ldots, X_{t_n}-X_{t_{n-1}}$ are all independent. 

By stationary,  for any $0\leq s < t <\infty$, the probability distribution of $X_t-X_s$ is the same as $X_{t-s}$.

One remarkable property of the L\'evy process is that any L\'evy process has a specific form of the characteristic function, called \textit{L\'evy-Khintchine formula}
\begin{equation}
    \mathbb{E}(e^{\ii (\xi, X_t)}) = e^{t\eta(\xi)}
\end{equation}
where
\begin{equation}\label{equ:eta}
    \eta(\xi) = \ii (b, \xi) - \frac{1}{2}(\xi, a\xi) + \int_{\RR^d\backslash \{0\}} \left[ e^{\ii (\xi, \by)}-1-\ii (\xi, \by)\mathbf{1}_{0<|\by|<1}(\by) \right]d\by
\end{equation}
here $b\in \RR^d$, $a$ is a positive definite symmetric matrix in $\RR^{d\times d}$, and $\nu$ is a L\'evy measure which satisfies
\begin{equation}
    \int_{\RR^d\backslash \{0\}} \min\{1, |\by|^2\}\nu(d\by)<\infty
\end{equation}

In the case $\nu\equiv 0$, we obtain the Gaussian process. In the case $\int_{\RR^d} \left[ e^{\ii (\xi, \by)}-1\right]dy$ is well defined, we can omit the term $\ii (\xi, \by)\mathbf{1}_{0<|\by|<1}(y)$. 

In the case $\nu<\infty$, the L\'evy process has the decomposition
\begin{equation}
    X_t = bt + \sqrt{a}B_t + \sum_{0\leq s\leq t} J_s 
\end{equation}
where $J_s$ is the jump at time $s$. To be precise, define 
\begin{equation}
    N(t, A) = \#\{0\leq s\leq t: J_s\in A\}
\end{equation}
if $t$ and $A$ is fixed, $N(t, A)$ is a random variable; if $t$ and $w\in \Omega$ is fixed, $N(t, \cdot)(w)$ is a measure; if $A$ is fixed, $N(\cdot, A)$ is a Poisson process with intensity $\nu(A)$. Therefore, we can also write
\begin{equation}
    \sum_{0\leq s\leq t} J_s = \int_{\RR^d-\{0\}} \bx N(t, d\bx)
\end{equation}

To end this section, we provide a third view of the L\'evy process. Consider the semigroup 
\begin{equation}
    (T_tf)(\bx) = \mathbb{E}(f(X_t+\bx))
\end{equation}
Then the infinitesimal generator will have the form
\begin{multline}\label{equ:Af}
    (Af)(\bx) =  b^i(\partial_i f)(\bx) + \frac{1}{2}a^{ij}(\partial_i \partial_j f)(\bx) +\\
    \int_{\RR^d\backslash\{0\}} [f(\bx+\by)-f(\bx)-\by\cdot(\nabla f)(\bx) \mathbf{1}_{0<|\by|<1}(\by)]\nu(d\by)
\end{multline}

\begin{remark}
    Another definition of the infinitesimal generator is through the Fourier transform 
    \begin{equation}
        (Af)(x) = \lim_{t\rightarrow 0+} \frac{P_t f-f}{t}
    \end{equation}
    where $P_tf = f\star p_t$ and $\mathcal{F}p_t(\xi) = e^{-t\eta(\xi)}$. 
    
    To see this, consider the case and without the adjustment term $-\by\cdot(\nabla f)(\bx) \mathbf{1}_{0<|\by|<1}(\by)$. By taking the Fourier transform of $(Af)(\bx)$, we have 
    \begin{multline}\label{equ:f}
        \mathcal{F}(Af)(\xi) = \ii\xi (b,\xi) f(\xi ) - \frac{1}{2}(a, a\xi)\hat f(\xi ) + \int_{\RR^d} {(\hat f(\xi ){e^{\ii \by\xi }} - \hat f(\xi ))\nu (\by)d\by} \\
        =\left( {\ii(b,\hat\xi) - \frac{1}{2}(\xi, a\xi) + \int_{\RR^d}{({e^{\ii (\by,\xi) }} - 1)\nu (\by)d\by} } \right)\hat f(\xi )
    \end{multline}
    this is exactly the expression we see in \cref{equ:eta}. 
    
    One the other hand, 
    \[\mathcal{F}\left[\mathop {\lim }\limits_{t \to 0 + } \frac{{{P_t}f - f}}{t}\right] = \mathop {\lim }\limits_{t \to 0 + } \frac{{{e^{\eta (\xi )t}}\hat f(\xi ) - \hat f(\xi )}}{t} = \eta (\xi )\]
    which coincides with \cref{equ:f}. 
    
\end{remark}

Let $T_t$ be the semigroup associated with the L\'evy process, and the associated infinitesimal generator is
\begin{multline}\label{equ:Af2}
    (Af)(\bx) =  c(\bx) f(\bx) + b^i(\partial_i f)(\bx) + \frac{1}{2}a^{ij}(\partial_i \partial_j f)(\bx) +\\
    \int_{\RR^d\backslash\{0\}} [f(\bx+\by)-f(\bx)-\by\cdot(\nabla f)(\bx) \mathbf{1}_{0<|\by|<1}(\by)]\nu(d\by)
\end{multline}

 we consider the transition measures $p_t(\bx)$ associated with $T_t$. Here $p_t$ is absolutely continuous with respect to Lebesgue measure. Define the adjoint operator $A^*$ of $A$, which satisfies
 \begin{equation}
     \int_{\RR^d} (Af)(\by) p_t(\by)d\by=\int_{\RR^d} f(\by)A^*p_t(\by)d\by
 \end{equation}
 for all $f\in C_c^\infty(\RR^d)$.

 In general, there is no nice form for $A^*$. However, in the case $c(\bx)$, $b(\bx)$ and $a(\bx)$ are all constant, we have
 \begin{multline}
     A^*p_t(\bx) = c (\partial_i p_t)(\bx) - b^i (\partial_i p_t)(\bx) + \frac{1}{2}a^{ij}\partial_i\partial_j p_t +\\
     \int_{\RR^d}\left[ p_t( \bx-\by) -p_t(\bx)+\by\cdot(\nabla p_t)(\bx) \mathbf{1}_{0<|\by|<1}(\by) \right]\nu(d\by)
 \end{multline}
 
 The Fokker-Planck equation, or Kolmogorov forward equation, is~\cite{sun2012fokker}
 \begin{equation}
     \frac{\partial p_t(\bx)}{\partial t} = A^*p_t(\bx)\quad p_0(\bx) = \delta(\bx)
 \end{equation}

\lib{$\mathcal{H}$ Matrix \label{sect:hmat}}

For completeness, we review the hierarchical matrix technique. For a comprehensive treatment of the $\mathcal{H}$-matrix, refer to \cite{borm2003introduction,bebendorf2008hierarchical,yang2008construction}. Especially we give a detailed description on the storage format, construction, fast matrix-vector multiplication routine, and LU decomposition. We later show how to construct the $\mathcal{H}$ matrix from kernels. 

The discretization of the jump-diffusion part $\int_\RR (u(x+y)-u(y))\nu(y)dy$ will usually lead to a dense matrix, which typically requires $\mathcal{O}(N^2)$ storage and has $\mathcal{O}(N^2)$ complexity for matrix-vector multiplication, $\mathcal{O}(N^3)$ for LU decomposition. Many techniques, such as the panel clustering methods and the fast multipole methods were developed. Later $\mathcal{H}$-matrix was considered by W.~Hackbusch, and many variations of hierarchical matrices have been intensively studied by researchers. $\mathcal{H}$-matrices can reduce the storage and arithmetics to nearly optimal complexity $\mathcal{O}(N)$ up to $\log N$ scaling. It relies on the fact that the kernel functions are smooth in the off-diagonal. 

\subsection{Construction and Storage}\label{equ:cons}

The construction of the $\mathcal{H}$ matrices can be best described in terms of matrix indices and the geometric points. Each entry $A_{ij}$ represents the interaction between two nodes $x_i$ and $x_j$. Let $I$, $J\subset \mathbb{N}$ be row and column index sets, then $A_{IJ}=(a_{ij})_{i\in I, j\in J}$ describes the interaction between a cluster $X_I = \{x_i\}_{i\in I}$ and another cluster $X_J = \{x_j\}_{j\in I}$. The interaction kernel function $k(x,y)$ is assumed to be smooth for sufficiently large $|x-y|$. 

Typically, it requires $\mathcal{O}(|I||J|)$ complexity to store the interaction data. However, if we assume that $I\subset J=\emptyset$ and geometrically the clusters $X_I$, $X_J$ are separate in the sense of admissibility, the cost can be reduced. 

\begin{definition}
 For two sets of indices $I$ and $J$ and the associated cluster $X_I$, $X_J$; assume that the kernel is asymptotically smooth, the admissibility condition is given by
    \begin{equation}\label{equ:adm}
        \min\{\mathrm{diam}(X_I),\mathrm{diam}(X_J) \}\leq \eta \mathrm{dist}(X_I,X_J)
    \end{equation}
    where $\mathrm{diam}(X_I) = \max_{x_i,x_j\in X}|x_i-x_j|$ AND $\mathrm{dist}(x_i,x_j)=\min_{x_i\in X_I, x_j\in X_J}|x_i-x_j|$. If the condition \cref{equ:adm} is not satisfied, we say $X_I$ and $X_J$ or $I$ and $J$ are inadmissible. 
\end{definition}

In our numerical examples, we use $\eta=1$, which indicates adjacent clusters are inadmissible since the distance is always zero. 

The admissible blocks usually have low rank structures. This is best illustrated by an example. Suppose $k(x,y)=\frac{1}{|x-y|^2}$, and further assume $x\in X_I$, $y\in X_J$. Assume $X_I$ and $Y_J$ are inadmissible, and $\bar x\in \mathcal{X}$, where $\mathcal{X}$ is the convex hull of $X_I$. Then we have
\[\frac{1}{{|x - y{|^2}}} = \frac{1}{{|x - \bar x - (y - \bar x){|^2}}} = \frac{1}{{|y - \bar x{|^2}{{\left| {\frac{{x - \bar x}}{{y - \bar x}} - 1} \right|}^2}}}\]

Since 
\[\left| {\frac{{x - \bar x}}{{y - \bar x}}} \right| < 1\]
we have 
\[\frac{1}{{|y - \bar x{|^2}{{\left| {\frac{{x - \bar x}}{{y - \bar x}} - 1} \right|}^2}}} = \frac{1}{{|y - \bar x{|^2}}}\left( {1 + \frac{{x - \bar x}}{{y - \bar x}} + {{\left( {\frac{{x - \bar x}}{{y - \bar x}}} \right)}^2} +  \ldots } \right)\]
which is in the form of 
\begin{equation}
    \frac{1}{{|x - y{|^2}}} = \sum_{n=0}^\infty \alpha_n(x-\bar x)\beta_n(y-\bar x)
\end{equation}

Then the series is convergent, and therefore the residual term will decay. It is possible to approximate $\frac{1}{{|x - y{|^2}}}$ with a few terms
\begin{equation}
    \frac{1}{{|x - y{|^2}}} \approx \sum_{n=0}^r \alpha_n(x-\bar x)\beta_n(y-\bar x)
\end{equation}
And therefore the interaction matrix for the cluster $A_{IJ}$ is 
\begin{equation}
    A_{IJ} = \left(\frac{1}{{|x_i - x_j{|^2}}}\right)_{i\in I, j\in J} =\left( \sum_{n=0}^r \alpha_n(x_i-\bar x)\beta_n(x_j-\bar x)\right)_{i\in I, j\in J} = UV'
\end{equation}
where
\begin{equation}
    U = \begin{bmatrix}
        \alpha_0(x_i-\bar x)& \alpha_1(x_i-\bar x)&\ldots & \alpha_r(x_i-\bar x)
    \end{bmatrix}
\end{equation}
\begin{equation}
    V = \begin{bmatrix}
        \beta_0(x_i-\bar x)& \beta_1(x_i-\bar x)&\ldots & \beta_r(x_i-\bar x)
    \end{bmatrix}
\end{equation}
If $r\ll |I|\wedge |J|$, we have achieved matrix compression using a low rank representation. 

The idea of the hierarchical matrix is then to classify each block $A_{IJ}$ into three types
\begin{itemize}
\item Full matrix. In this case, $A_{IJ}$ is represented using fully populated matrices.
\item Low-rank matrix. In the case $I$ and $J$ are admissible, we can store the block $A_{IJ}$ in the form of low-rank matrices. This will help us save storage and computational cost.
    \item $\mathcal{H}$-matrix. For the blocks that are neither low-rank matrix nor small enough to become a full matrix, it is further divided into sub-blocks~(for example, via quadtree structure). 
\end{itemize}

The $\mathcal{H}$-matrix will be stored in a hierarchical format, and there exist three kinds of sub-blocks
\begin{itemize}
    \item Full dense blocks. These blocks cannot be represented as the low-rank block and cannot be subdivided because its size is smaller than a pre-assigned threshold.
    \item Low-rank blocks. These blocks are stored using low-rank factorizations. Note low-rank blocks can also be dense blocks. We have abused the terminology here, but it can be easily figured out from context.
    \item Hierarchical blocks, or $\mathcal{H}$-blocks. These blocks do not have low-rank factorization, but their sizes are so large that they can be further subdivided into new blocks.
\end{itemize}
The hyper-parameters we need to decide on the $\mathcal{H}$-matrix construction are
\begin{itemize}
    \item The minimum block size $N_{\min}$. It defined the minimum block size we can have for the sub-blocks in the $\mathcal{H}$-block. This indicates that if a $N\times N$ matrix is not a low-rank matrix, where $N\leq N_{\min}$, we should store this sub-block in the format of a dense matrix.
    \item The maximum block size $N_{\max}$. It defines the maximum sub-block size we can have. For convenience, we define it in terms of $N_{\mathrm{block}}:=\left\lceil\frac{N}{N_{\max}}\right\rceil$, where $N$ is the matrix dimension.
\end{itemize}

Typically, $N_{\min}=64$ is a good choice to exploit the efficient dense linear algebra provided by LAPACK/BLAS. In addition, $N_{\mathrm{block}}=4$ or $8$ are good empirical choices.

\subsection{Matrix Vector Multiplication}

One advantage of the $\mathcal{H}$ matrix is that the matrix-vector multiplication is cheap. The matrix-vector multiplication of $\mathcal{H}$-matrix can be described through the rule of the operator for three different kinds of sub-blocks
\begin{itemize}
    \item Full matrix. In this case, the normal dense matrix-vector multiplication is used. 
    \item Low-rank matrix. The operator can be carried out quite efficiently via
    \begin{equation}
        (UV')x = U(V'x)
    \end{equation}
    note $V'x$ is a $r\times 1$ vector.
    \item $\mathcal{H}$-matrix. If the sub-block is 
    \begin{equation}
        B = \begin{bmatrix}
            B_{11}&B_{12}\\
            B_{21}&B_{22}
        \end{bmatrix}\quad x = \begin{bmatrix}
            x_1\\
            x_2
        \end{bmatrix}
    \end{equation}
    the matrix vector multiplication will be carried out recursively, i.e.
    \begin{equation}
        Bx = \begin{bmatrix}
            B_{11} x_1 + B_{12}x_2\\
            B_{21} x_1 + B_{22}x_2
        \end{bmatrix}
    \end{equation}
\end{itemize}

\subsection{LU Decomposition}

$\mathcal{H}$-LU can be done in $\mathcal{H}$-matrix format and recursively in computational cost $\mathcal{O}(N)$ up to a $\log N$ scaling compared to dense LU in $\mathcal{O}(N^3)$.

We need to define a triangular solver which solves $AX=B$ for lower triangular matrix or $XA = B$ or upper triangular matrix. The matrices are either $\mathcal{H}$-matrix or full matrix. We only need to consider the lower triangular cases since in the latter case by transposition $A'X'=B'$; we reduce the problem to the former. 

The triangular solver will work differently for different situations.
\begin{itemize}
    \item If $B$ is a full matrix, then $X$ is a full matrix and $X=A^{-1}B$. Here $A$ is converted to a full matrix.
    \item If $B$ is a low rank matrix, $B=B_1B_2'$, then $X$ is also a low rank matrix $X = (A^{-1}B_1)B_2'$.
    \item If $A$ and $B$ are both hierarchical matrices
    \[\left[ {\begin{array}{*{20}{c}}
{{A_{11}}}&{}\\
{{A_{21}}}&{{A_{22}}}
\end{array}} \right]\left[ {\begin{array}{*{20}{c}}
{{X_{11}}}&{{X_{12}}}\\
{{X_{21}}}&{{X_{22}}}
\end{array}} \right] = \left[ {\begin{array}{*{20}{c}}
{{B_{11}}}&{{B_{12}}}\\
{{B_{21}}}&{{B_{22}}}
\end{array}} \right]\]
Then we will first solve $A_{11}X_{11}=B_{11}$ and $A_{11}X_{12}=B_{12}$. Then we solve
\begin{equation}
    A_{22}X_{21} = B_{21}-A_{21}X_{11}\quad A_{22}X_{22} = B_{22}-A_{21}X_{12}
\end{equation}    
\end{itemize}

The LU decomposition also works differently for different types of matrices. Again only full matrices and $\mathcal{H}$ matrices are considered. 

For full matrices, the standard dense LU is adopted. For $\mathcal{H}$ matrices, 
\[\left[ {\begin{array}{*{20}{c}}
{{A_{11}}}&{{A_{12}}}\\
{{A_{21}}}&{{A_{22}}}
\end{array}} \right] = \left[ {\begin{array}{*{20}{c}}
{{L_{11}}}&\\
{{L_{21}}}&{{L_{22}}}
\end{array}} \right]\left[ {\begin{array}{*{20}{c}}
{{U_{11}}}&{{U_{12}}}\\
{}&{{U_{22}}}
\end{array}} \right]\]

The algorithm will work as follows
\begin{itemize}
    \item LU decomposition of $A_{11}=L_{11}U_{11}$
    \item Triangular solve $L_{11}U_{12}=A_{12}$~(lower triangular, $U_{12}$ is the unknown)
    \item Triangular solve $L_{21}U_{11} = A_{21}$~(upper triangular, $L_{21}$ is the unknown)
    \item LU decomposition of $A_{22}-L_{21}U_{12}=L_{22}U_{22}$
\end{itemize}    

The LU decomposition can also be performed in an in-place way, which will save storage.

\section{Proof of \Cref{lemma:glemma}}\label{sect:proof}

Note that 
    \begin{align}
        &|{e^{ - {\varepsilon ^2}{{(x - y)}^2}}} - \sum\limits_{n = 0}^r {{\alpha _n}} (x - \bar x){\beta _n}(y - \bar x)|\\
         =&  {e^{ - {\varepsilon ^2}{t^2} - {\varepsilon ^2}t_0^2}}\left( {\frac{{{{(2{\varepsilon ^2}{t_0}t)}^{n + 1}}}}{{(n + 1)!}} + \frac{{{{(2{\varepsilon ^2}{t_0}t)}^{n + 2}}}}{{(n + 2)!}} +  \ldots } \right)\\
        \le& {e^{ - {\varepsilon ^2}{t^2} - {\varepsilon ^2}t_0^2}}\frac{{{{(2{\varepsilon ^2}{t_0}t)}^{n + 1}}}}{{(n + 1)!}}{e^{2{\varepsilon ^2}{t_0}t}} \le {e^{2{\varepsilon ^2}{D^2}}}\frac{{{{(2{\varepsilon ^2}{D^2})}^{n + 1}}}}{{(n + 1)!}}
    \end{align}
    We invoke the basic estimate
    \begin{equation}
        n!>\left( \frac{n}{3} \right)^{n}
    \end{equation} 
    and obtain
    \begin{equation}
        |{e^{ - {\varepsilon ^2}{{(x - y)}^2}}} - \sum\limits_{n = 0}^r {{\alpha _n}} (x - \bar x){\beta _n}(y - \bar x)|\leq {e^{2{\varepsilon ^2}{D^2}}}\frac{{{{(2{\varepsilon ^2}{D^2})}^{n + 1}}}}{{{{\left( {\frac{{n + 1}}{3}} \right)}^{n + 1}}}} = {\left( {\frac{{6{\varepsilon ^2}{D^2}}}{{n + 1}}} \right)^{n + 1}}{e^{2{\varepsilon ^2}{D^2}}}
    \end{equation}
    
    Since we have \cref{equ:r}, which indicates 
    \begin{equation}
        {\frac{{6{\varepsilon ^2}{D^2}}}{{n + 1}}}<\frac{1}{2}
    \end{equation}
    and therefore
    \begin{equation}
        |{e^{ - {\varepsilon ^2}{{(x - y)}^2}}} - \sum\limits_{n = 0}^r {{\alpha _n}} (x - \bar x){\beta _n}(y - \bar x)| < {\left( {\frac{1}{2}} \right)^{n + 1}}{e^{2{\varepsilon ^2}{D^2}}} < \delta 
    \end{equation}
    the last equation is due to the assumption \cref{equ:r}.

\end{appendices}

\bibliographystyle{mybib}
\bibliography{levy.bib}

\begin{thebibliography}{10}

\bibitem{gatto2015numerical}
P.~Gatto and J.~S. Hesthaven.
\newblock \titlecap{Numerical Approximation of the Fractional Laplacian via
  hp-finite Elements, with an Application to Image Denoising}.
\newblock {\em Journal of Scientific Computing}, 65(1):249--270, 2015.

\bibitem{unser2009multiresolution}
M.~Unser, D.~Sage, and D.~Van De~Ville.
\newblock \titlecap{Multiresolution monogenic signal analysis using the
  Riesz--Laplace wavelet transform}.
\newblock {\em IEEE Transactions on Image Processing}, 18(11):2402--2418, 2009.

\bibitem{scalas2000fractional}
E.~Scalas, R.~Gorenflo, and F.~Mainardi.
\newblock \titlecap{Fractional calculus and continuous-time finance}.
\newblock {\em Physica A: Statistical Mechanics and its Applications},
  284(1-4):376--384, 2000.

\bibitem{bogdan2003censored}
K.~Bogdan, K.~Burdzy, and Z.-Q. Chen.
\newblock \titlecap{Censored stable processes}.
\newblock {\em Probability theory and related fields}, 127(1):89--152, 2003.

\bibitem{zaslavsky2002chaos}
G.~M. Zaslavsky.
\newblock \titlecap{Chaos, fractional kinetics, and anomalous transport}.
\newblock {\em Physics Reports}, 371(6):461--580, 2002.

\bibitem{alibaud2012continuous}
N.~Alibaud, S.~Cifani, and E.~R. Jakobsen.
\newblock \titlecap{Continuous dependence estimates for nonlinear fractional
  convection-diffusion equations}.
\newblock {\em SIAM Journal on Mathematical Analysis}, 44(2):603--632, 2012.

\bibitem{chen2010anomalous}
W.~Chen, H.~Sun, X.~Zhang, and D.~Koro{\v{s}}ak.
\newblock \titlecap{Anomalous diffusion modeling by fractal and fractional
  derivatives}.
\newblock {\em Computers \& Mathematics with Applications}, 59(5):1754--1758,
  2010.

\bibitem{chen2006speculative}
W.~Chen.
\newblock \titlecap{A speculative study of 2/ 3-order fractional Laplacian
  modeling of turbulence: Some thoughts and conjectures}.
\newblock {\em Chaos: An Interdisciplinary Journal of Nonlinear Science},
  16(2):023126, 2006.

\bibitem{epps2018turbulence}
B.~P. Epps and B.~Cushman-Roisin.
\newblock \titlecap{Turbulence Modeling via the Fractional Laplacian}.
\newblock {\em arXiv preprint arXiv:1803.05286}, 2018.

\bibitem{lischke2018fractional}
A.~Lischke, G.~Pang, M.~Gulian, F.~Song, C.~Glusa, X.~Zheng, Z.~Mao, W.~Cai,
  M.~M. Meerschaert, M.~Ainsworth, et~al.
\newblock \titlecap{What Is the Fractional Laplacian?}
\newblock {\em arXiv preprint arXiv:1801.09767}, 2018.

\bibitem{bonito2018numerical}
A.~Bonito, J.~P. Borthagaray, R.~H. Nochetto, E.~Otarola, and A.~J. Salgado.
\newblock \titlecap{Numerical methods for fractional diffusion}.
\newblock {\em Computing and Visualization in Science}, pages 1--28, 2018.

\bibitem{huang2016finite}
Y.~Huang and A.~Oberman.
\newblock \titlecap{Finite difference methods for fractional Laplacians}.
\newblock {\em arXiv preprint arXiv:1611.00164}, 2016.

\bibitem{meidner2017hp}
D.~Meidner, J.~Pfefferer, K.~Sch{\"u}rholz, and B.~Vexler.
\newblock \titlecap{hp-Finite Elements for Fractional Diffusion}.
\newblock {\em arXiv preprint arXiv:1706.04066}, 2017.

\bibitem{kyprianou2017unbiased}
A.~E. Kyprianou, A.~Osojnik, and T.~Shardlow.
\newblock \titlecap{Unbiased ‘walk-on-spheres’ Monte Carlo methods for the
  fractional Laplacian}.
\newblock {\em IMA Journal of Numerical Analysis}, 2017.

\bibitem{acosta2018regularity}
G.~Acosta, J.~P. Borthagaray, O.~Bruno, and M.~Maas.
\newblock \titlecap{Regularity theory and high order numerical methods for the
  (1D)-fractional Laplacian}.
\newblock {\em Mathematics of Computation}, 87(312):1821--1857, 2018.

\bibitem{zhao2017adaptive}
X.~Zhao, X.~Hu, W.~Cai, and G.~E. Karniadakis.
\newblock \titlecap{Adaptive finite element method for fractional differential
  equations using hierarchical matrices}.
\newblock {\em Computer Methods in Applied Mechanics and Engineering},
  325:56--76, 2017.

\bibitem{garbaczewski2018fractional}
P.~Garbaczewski.
\newblock \titlecap{Fractional Laplacians and L\'evy flights in bounded
  domains}.
\newblock {\em arXiv preprint arXiv:1802.09853}, 2018.

\bibitem{barndorff2012levy}
O.~E. Barndorff-Nielsen, T.~Mikosch, and S.~I. Resnick.
\newblock {\em L{\'e}vy processes: theory and applications}.
\newblock Springer Science \& Business Media, 2012.

\bibitem{kwasnicki2017ten}
M.~Kwa{\'s}nicki.
\newblock \titlecap{Ten equivalent definitions of the fractional Laplace
  operator}.
\newblock {\em Fractional Calculus and Applied Analysis}, 20(1):7--51, 2017.

\bibitem{MichaelC25:online}
\titlecap{A survey of numerical methods for L\'evy Markets}.
\newblock
  \url{ftp://www.cs.toronto.edu/na/reports/Michael.Chiu.MEng.Project.pdf}.
\newblock (Accessed on 10/24/2018).

\bibitem{bebendorf2008hierarchical}
M.~Bebendorf.
\newblock {\em Hierarchical matrices}.
\newblock Springer, 2008.

\bibitem{borm2003introduction}
S.~B{\"o}rm, L.~Grasedyck, and W.~Hackbusch.
\newblock \titlecap{Introduction to hierarchical matrices with applications}.
\newblock {\em Engineering analysis with boundary elements}, 27(5):405--422,
  2003.

\bibitem{golub2012matrix}
G.~H. Golub and C.~F. Van~Loan.
\newblock {\em Matrix computations}, volume~3.
\newblock JHU Press, 2012.

\bibitem{hatzinikitas2009fractional}
A.~Hatzinikitas.
\newblock \titlecap{The fractional Schr{\"o}dinger operator and Toeplitz
  matrices}.
\newblock {\em Journal of Mathematical Physics}, 50(10):103524, 2009.

\bibitem{chen2016two}
H.~Chen, H.~Zhou, Q.~Li, and Y.~Wang.
\newblock \titlecap{Two efficient modeling schemes for fractional Laplacian
  viscoacoustic wave equation}.
\newblock {\em Geophysics}, 81(5):T233--T249, 2016.

\bibitem{bailey1994fast}
D.~H. Bailey and P.~N. Swarztrauber.
\newblock \titlecap{A fast method for the numerical evaluation of continuous
  Fourier and Laplace transforms}.
\newblock {\em SIAM Journal on Scientific Computing}, 15(5):1105--1110, 1994.

\bibitem{chaturapruek2013crime}
S.~Chaturapruek, J.~Breslau, D.~Yazdi, T.~Kolokolnikov, and S.~G. McCalla.
\newblock \titlecap{Crime modeling with L{\'e}vy flights}.
\newblock {\em SIAM Journal on Applied Mathematics}, 73(4):1703--1720, 2013.

\bibitem{massei2018fast}
S.~Massei, M.~Mazza, and L.~Robol.
\newblock \titlecap{Fast solvers for 2D fractional diffusion equations using
  rank structured matrices}.
\newblock {\em arXiv preprint arXiv:1804.05522}, 2018.

\bibitem{karkulikh}
M.~Karkulik and J.~Melenk.
\newblock \titlecap{H-matrix approximability of inverses of discretizations of
  the fractional Laplacian}.

\bibitem{metropolis1989monte}
N.~Metropolis.
\newblock \titlecap{Monte Carlo Method}.
\newblock {\em From Cardinals to Chaos: Reflection on the Life and Legacy of
  Stanislaw Ulam}, page 125, 1989.

\bibitem{hammersley2013monte}
J.~Hammersley.
\newblock {\em Monte carlo methods}.
\newblock Springer Science \& Business Media, 2013.

\bibitem{tankov2009jump}
P.~Tankov and E.~Voltchkova.
\newblock \titlecap{Jump-diffusion models: a practitioner’s guide}.
\newblock {\em Banque et March{\'e}s}, 99(1):24, 2009.

\bibitem{kou2002jump}
S.~G. Kou.
\newblock \titlecap{A jump-diffusion model for option pricing}.
\newblock {\em Management science}, 48(8):1086--1101, 2002.

\bibitem{laskin2010principles}
N.~Laskin.
\newblock \titlecap{Principles of fractional quantum mechanics}.
\newblock {\em arXiv preprint arXiv:1009.5533}, 2010.

\bibitem{hasan2018tunneling}
M.~Hasan and B.~P. Mandal.
\newblock \titlecap{Tunneling time in space fractional quantum mechanics}.
\newblock {\em Physics Letters A}, 382(5):248--252, 2018.

\bibitem{garbaczewski1995schrodinger}
P.~Garbaczewski, J.~R. Klauder, and R.~Olkiewicz.
\newblock \titlecap{Schr{\"o}dinger problem, L{\'e}vy processes, and noise in
  relativistic quantum mechanics}.
\newblock {\em Physical Review E}, 51(5):4114, 1995.

\bibitem{laskin2000fractional}
N.~Laskin.
\newblock \titlecap{Fractional quantum mechanics and L{\'e}vy path integrals}.
\newblock {\em Physics Letters A}, 268(4-6):298--305, 2000.

\bibitem{cont2005finite}
R.~Cont and E.~Voltchkova.
\newblock \titlecap{A finite difference scheme for option pricing in jump
  diffusion and exponential L{\'e}vy models}.
\newblock {\em SIAM Journal on Numerical Analysis}, 43(4):1596--1626, 2005.

\bibitem{matsuda2004introduction}
K.~Matsuda.
\newblock \titlecap{Introduction to Merton jump diffusion model}.
\newblock {\em Department of Economics. The Graduate Center, The City
  University of New York}, 2004.

\bibitem{zhao2005adaptive}
K.~Zhao, M.~N. Vouvakis, and J.-F. Lee.
\newblock \titlecap{The adaptive cross approximation algorithm for accelerated
  method of moments computations of EMC problems}.
\newblock {\em IEEE transactions on electromagnetic compatibility},
  47(4):763--773, 2005.

\bibitem{fong2009black}
W.~Fong and E.~Darve.
\newblock \titlecap{The black-box fast multipole method}.
\newblock {\em Journal of Computational Physics}, 228(23):8712--8725, 2009.

\bibitem{l2016hierarchical}
K.~L.~Ho and L.~Ying.
\newblock \titlecap{Hierarchical interpolative factorization for elliptic
  operators: integral equations}.
\newblock {\em Communications on Pure and Applied Mathematics},
  69(7):1314--1353, 2016.

\bibitem{isaacson2012analysis}
E.~Isaacson and H.~B. Keller.
\newblock {\em Analysis of numerical methods}.
\newblock Courier Corporation, 2012.

\bibitem{giles2005convergence}
M.~B. Giles and R.~Carter.
\newblock \titlecap{Convergence analysis of Crank-Nicolson and Rannacher
  time-marching}.
\newblock Technical report, Unspecified, 2005.

\bibitem{minden2018simple}
V.~Minden and L.~Ying.
\newblock \titlecap{A simple solver for the fractional Laplacian in multiple
  dimensions}.
\newblock {\em arXiv preprint arXiv:1802.03770}, 2018.

\bibitem{jarvenpaa2006singularity}
S.~Jarvenpaa, M.~Taskinen, and P.~Yla-Oijala.
\newblock \titlecap{Singularity subtraction technique for high-order polynomial
  vector basis functions on planar triangles}.
\newblock {\em IEEE transactions on antennas and propagation}, 54(1):42--49,
  2006.

\bibitem{wilde1998direct}
A.~Wilde and M.~Aliabadi.
\newblock \titlecap{Direct evaluation of boundary stresses in the 3D BEM of
  elastostatics}.
\newblock {\em Communications in Numerical Methods in Engineering},
  14(6):505--517, 1998.

\bibitem{hanninen2006singularity}
I.~Hanninen, M.~Taskinen, and J.~Sarvas.
\newblock \titlecap{Singularity subtraction integral formulae for surface
  integral equations with RWG, rooftop and hybrid basis functions}.
\newblock {\em Progress In Electromagnetics Research}, 63:243--278, 2006.

\bibitem{anselone1981singularity}
P.~Anselone.
\newblock \titlecap{Singularity subtraction in the numerical solution of
  integral equations}.
\newblock {\em The ANZIAM Journal}, 22(4):408--418, 1981.

\bibitem{ros2014dirichlet}
X.~Ros-Oton and J.~Serra.
\newblock \titlecap{The Dirichlet problem for the fractional Laplacian:
  regularity up to the boundary}.
\newblock {\em Journal de Math{\'e}matiques Pures et Appliqu{\'e}es},
  101(3):275--302, 2014.

\bibitem{coulier2017inverse}
P.~Coulier, H.~Pouransari, and E.~Darve.
\newblock \titlecap{The inverse fast multipole method: using a fast approximate
  direct solver as a preconditioner for dense linear systems}.
\newblock {\em SIAM Journal on Scientific Computing}, 39(3):A761--A796, 2017.

\bibitem{pouransari2017fast}
H.~Pouransari, P.~Coulier, and E.~Darve.
\newblock \titlecap{Fast hierarchical solvers for sparse matrices using
  extended sparsification and low-rank approximation}.
\newblock {\em SIAM Journal on Scientific Computing}, 39(3):A797--A830, 2017.

\bibitem{doney2007introduction}
R.~A. Doney.
\newblock \titlecap{Introduction to L{\'e}vy Processes}.
\newblock {\em Fluctuation Theory for L{\'e}vy Processes: Ecole d'Et{\'e} de
  Probabilit{\'e}s de Saint-Flour XXXV-2005}, pages 1--8, 2007.

\bibitem{papapantoleon2008introduction}
A.~Papapantoleon.
\newblock \titlecap{An introduction to L\'evy processes with applications in
  finance}.
\newblock {\em arXiv preprint arXiv:0804.0482}, 2008.

\bibitem{sato1999levy}
K.-i. Sato and S.~Ken-Iti.
\newblock {\em L{\'e}vy processes and infinitely divisible distributions}.
\newblock Cambridge university press, 1999.

\bibitem{sun2012fokker}
X.~Sun and J.~Duan.
\newblock \titlecap{Fokker-Planck equations for nonlinear dynamical systems
  driven by non-Gaussian L{\'e}vy processes}.
\newblock {\em Journal of Mathematical Physics}, 53(7):072701, 2012.

\bibitem{yang2008construction}
F.~Yang.
\newblock \titlecap{Construction and application of hierarchical matrix
  preconditioners}.
\newblock 2008.

\end{thebibliography}
\end{document}